\documentclass[10pt,a4paper]{article}

\usepackage{amsmath,amsthm,amsopn,amsfonts,amssymb,dsfont,color}
\usepackage[dvips]{graphicx}
\usepackage{psfrag}

\usepackage{a4,a4wide}

\def\ep{{\varepsilon}}
\def\R{\mathbb R}
\def\N{\mathbb N}
\def\C{\mathbb C}

\newtheorem{theorem}{\textbf{Theorem}}[section]

\newtheorem{lemma}[theorem]{\textbf{Lemma}}
\newtheorem{proposition}[theorem]{\textbf{Proposition}}
\newtheorem{claim}[theorem]{\textbf{Claim}}
\newtheorem{corollary}[theorem]{\textbf{Corollary}}
\newtheorem{definition}[theorem]{\textbf{Definition}}
\newtheorem{assumption}[theorem]{\textbf{Assumption}}
\theoremstyle{remark}
\newtheorem{rem}[theorem]{\textbf{Remark}}

\title{
Travelling waves for a non-monotone bistable equation with delay: existence and oscillations}
\date{}

\begin{document}

\maketitle

\begin{center}
{\large\bf Matthieu Alfaro \footnote{ IMAG, Universit\'e de
Montpellier, CC051, Place Eug\`ene Bataillon, 34095 Montpellier
Cedex 5, France. E-mail: matthieu.alfaro@umontpellier.fr},  Arnaud Ducrot
\footnote{IMB UMR CNRS 5251, Universit\'e de Bordeaux,
3 ter, Place de la Victoire, 33000 Bordeaux, France. E-mail: arnaud.ducrot@u-bordeaux.fr} and Thomas Giletti \footnote{IECL, Universit\'{e} de Lorraine, B.P. 70239, 54506
Vandoeuvre-l\`{e}s-Nancy Cedex, France. E-mail:
thomas.giletti@univ-lorraine.fr}.}\\
[2ex]
\end{center}


 \tableofcontents

\vspace{10pt}

\begin{abstract} We consider a bistable ($0<\theta<1$ being the three constant steady states) delayed reaction diffusion equation, which serves as a model in population dynamics. The problem does not admit any comparison principle. This prevents the use of classical technics and, as a consequence, it is far from obvious to understand the behaviour of a possible travelling wave in $+\infty$. Combining refined {\it a priori} estimates and a Leray Schauder topological degree argument, we construct a travelling wave connecting 0 in $-\infty$ to \lq\lq something'' which is strictly above the unstable equilibrium $\theta$ in $+\infty$.  Furthemore, we present situations (additional bound on the nonlinearity or small delay) where the wave converges to 1 in $+\infty$, whereas the wave is shown to oscillate around 1 in $+\infty$ when, typically, the delay is large.\\


\end{abstract}

\section{Introduction}\label{s:intro}

We consider the following delayed reaction-diffusion equation
\begin{equation}\label{eq}
\partial _t u(t,x)=\partial _{xx} u(t,x)+f(u(t-\tau,x))-u(t,x),\quad t>0,\; x\in \R,
\end{equation}
where $\tau >0$ is a given time delay. This equation typically describes the spatio-temporal evolution of a population density $u=u(t,x)$ at time $t$ and location $x\in\R$. In this context $f$ describes the birth rate function while the term $-u$ represents the (normalized) death rate. The time delay $\tau$ allows to take into account a maturity period.

The existence and the properties of travelling wave solutions for problem \eqref{eq} have received a lot of attention in the case where the function $f$ is increasing, so that problem \eqref{eq} generates a monotone semiflow. The literature can  mostly be divided into the so-called monostable and bistable cases, the distinction 
being made with respect to the dynamical properties of the ODE problem
\begin{equation*}
\dot{u}=F(u):=f(u)-u.
\end{equation*}
The monostable case contains, as a special case, the so-called Ricker's function $f(u)=\beta ue^{-u}$ for some parameter $\beta>1$, for which \eqref{eq} is then referred to as the diffusive Nicholson's blowflies equation \cite{Nichol, Nichol1}. On the other hand , a typical bistable situation is given by $f(u)=\beta u^2e^{-u}$, for some  parameter $\beta>e$, see \cite{Ma}.

For the study of monostable monotone waves we refer to the works of So, Wu and Zou \cite{So}, Thieme and Zhao \cite{Thieme}, Gourley and So \cite{Gourley}, Li, Ruan and Wang \cite{Li}, Liang and Zhao \cite{Liang}, Ma \cite{MA}, and the references therein. Monotone waves for bistable problems of the form \eqref{eq} have also been investigated and we refer the reader to Schaaf \cite{Sch}, Smith and Zhao \cite{Smi-Zha}, Ma and Wu \cite{Ma}, Wang, Li and Ruan \cite{Wan-Li-Rua, Wang}, Fang and Zhao \cite{Fan-Zha2}. Notice also 
an analysis of the spreading properties of problem \eqref{eq}, in a bistable situation, by Lin and Ruan \cite{Lin}.

Note that the aforementioned works are concerned with monotone birth rate functions $f$ or, more generally, with monotone semiflows. Much less is known when this assumption is removed. Let us underline the analysis of a non-monotone monostable situation  by  Trofimchuk et al. \cite{Trofimchuk}, who prove the existence of travelling waves exhibiting oscillations on one side. 

In this work we shall discuss the existence of travelling wave solutions for problem \eqref{eq} when the birth rate function $f$, of the bistable type, is not assumed to be
monotone  (see below for precise statements). Our set of assumptions prevents the existence
of a comparison principle, and makes the analysis rather involved.

Roughly speaking we prove that, if the nonlinear function $F(u):=f(u)-u$ is bistable between $0$ and $1$ and $\int_0^1 F(u)du>0$, then problem \eqref{eq} admits a travelling wave solution with a positive wave speed.
To prove this result we make use of a Leray-Schauder topological degree argument to construct a solution of a similar problem posed on a bounded domain. Then, using refined estimates, we are able to pass to the limit as the length of the bounded domain tends to infinity. We thus
get a travelling wave in the sense of Definition
\ref{def:tw}, that is connecting $0$ in $-\infty$ to \lq\lq something'' which is above the intermediate equilibrium
in $+\infty$, meaning that the wave is truly bistable. Then, in some regimes, we are able to precise their
behaviours in $+\infty$: convergence to $1$ under an additional bound on $f$ or when the delay is small, 
but oscillations around $1$ for large delay (and an additional assumption on the shape of $f$).

\section{Assumptions and results}\label{s:results}

Throughout the paper, we make the following assumption on the bistable nonlinearity $f$, whose three fixed points are  $0<\theta<1$.

\begin{assumption}[Bistable nonlinearity]\label{ass:f} The function  $f:\R \to [0,\infty)$ is of the class $ C^{1,\gamma}$ on $[0,\infty)$ for some $\gamma >0$. There are 
$0< \theta <1$ such that
$$
f(u)=0 \text{ on } (-\infty,0), \quad f(u)<u \text{ on }
(0,\theta), \quad f(u)>u \text{ on } (\theta,1), \quad f(u)<u
\text{ on } (1,\infty).
$$
We also denote
$$
M:= \max_{[0,1]} f \geq 1,
$$
and assume that $f > \theta$ in $(\theta,M]$ (notice that this is automatically satisfied if $M=1$). 
Moreover
\begin{equation}
\label{ass:integrale}  \int_0^1 (f(u) -u)du >0.
\end{equation}
\end{assumption}

Let us emphasize again that when $f$ is monotone, the situation is well understood: there is a unique (up to translation) travelling wave solution $\left(c,U\right)\in\R\times C^2_b(\R)$, 
which moreover satisfies
\begin{equation*}
U(-\infty)=0,\;\ U(+\infty)=1,\text{ and }U'>0.
\end{equation*}
The existence follows from the results of Fang and Zhao  \cite{Fan-Zha2}, see also \cite{Sch}, while uniqueness is ensured by the results of Smith and Zhao \cite{Smi-Zha}. 

 In this work no monotony assumption on the nonlinearity is made. We do however assume $f\geq 0 = f(0)$, to insure the positivity of travelling waves, and $f>\theta=f(\theta)$ on $(\theta,M]$, to insure a truly bistable behaviour of travelling waves. Nonetheless our assumption allows the situation $M>1=f(1)$. In such a case,  oscillations around~1 may occur when the delay is large, see below. On the other hand, if either $M=1$ or the delay is small,  we will show that the travelling wave does converge to~1 at $+\infty$. Notice that our hypothesis allows the degenerate  situations $f'(0)=1$, $f'(\theta)=1$, $f'(1)=1$. 

{}From the mathematical point of view, assumption \eqref{ass:integrale}  insures that we can work with positive speeds. Hence, when
considering the travelling wave equation, the information is to be searched \lq\lq on the left''. If we cannot insure positivity of
 speeds, then the analysis becomes much more difficult (if possible). Notice that, from the modelling point of view, \eqref{ass:integrale} is hardly a
restriction since, having in mind a Allee effect, the unstable zero $\theta$ is usually small.

For simplicity we make the smoothness assumption $f\in C^{1,\gamma}([0,\infty))$ but this could be relaxed to the conditions that $f$ is Lipschitz and satisfies some kind of H\"{o}lder regularity at 0 and 1, that is the conditions used in \cite{Ber-Nir} whose some of the techniques will be used later on.

\begin{definition}[Bistable travelling wave]\label{def:tw} A (bistable) travelling wave for equation \eqref{eq} is a couple $(c,U)\in \R \times C^2(\R)$, with $U>0$ on $\R$, and such that
\begin{equation*}
\begin{cases}
-U''+cU'=f(U(\cdot -c\tau))-U \quad \text{ in } \R ,\\
U(-\infty)=0, \quad \liminf_{x\to+\infty}U(x)>\theta.
\end{cases}
\end{equation*}
\end{definition}

Notice that the boundary condition as $x\to +\infty$ is understood in a \lq\lq weak'' sense in Definition~\ref{def:tw}, which is quite classical when nonlocal effects that may destabilize the expected boundary condition (1 in the present case) are concerned; see \cite{Ber-Nad-Per-Ryz}, \cite{Alf-Cov-Rao1} in a KPP context or \cite{Alf-Cov-Rao2} in a bistable situation. 

Our first main result consists in constructing such a bistable travelling wave.

\begin{theorem}[Construction of a bistable travelling wave]\label{th:construction} Let Assumption \ref{ass:f} hold. Then, there exists a bistable travelling wave such that $c>0$, $U<M$, and $U\in
C^2_b(\R)$
, with the normalization $U(0)=\theta$. 

Furthermore, $U$ is increasing on $(-\infty,c\tau]$ and satisfies $ U > \theta$ on $[c\tau,+\infty)$.
\end{theorem}

In the sequel we enquire on the behaviour of travelling waves as $x\to +\infty$. Our next result provides two sufficient conditions under which the bistable travelling wave constructed in Theorem~\ref{th:construction} does converge to 1 in $+\infty$.

\begin{theorem}[Convergence to 1]\label{th:small-delay} Let $(c,U)\in\R\times C^2_b(\R)$ be a bistable travelling wave. Assume either
\begin{equation}
\label{add-bound}
M:=\max_{[0,1]} f  = 1 \quad \mbox{(additional bound on $f$)},
\end{equation}
or
\begin{equation}
\label{small-delay}
\tau\Vert f'\Vert _{L^\infty(0,\Vert U\Vert _\infty)}<1 \quad \mbox{(small delay)}.
\end{equation}
Then $U(+\infty)=1$.
\end{theorem}

As far as travelling waves of the nonlocal Fisher-KPP equation 
$$
\partial _t u=\partial _{xx} u+u(1-\phi *u)
$$
are concerned the convergence to 1 is known to hold not only for focused competition kernels $\phi$, see \cite{Ber-Nad-Per-Ryz}, but also for large speeds, and this even if the state 1 is driven unstable by the kernel $\phi$, see \cite{Fan-Zha} and \cite{Alf-Cov}. Also, for cases of convergence to 1, we refer to \cite{Gom-Tro} for a Fisher-KPP equation with delay and to \cite{Alf-Cov-Rao2} for a nonlocal bistable equation.

Now, we wonder if there are travelling waves that connect 0 in $-\infty$ to  a wavetrain oscillating around $1$  in $+\infty$. Such connections zero-wavetrain are known
to exist in some nonlocal equations. For instance, we refer to
\cite{Nad-Per-Ros-Ryz} for a simplified model which allows
to perform explicit computations. See also the works \cite{Trofimchuk}, \cite{ducrot}, \cite{Gom-Tro},  \cite{Has-Tro}, \cite{Duc-Nad} in KPP situations. In our bistable context, our last main result shows that such a connection zero-wavetrain can also exist:
roughly speaking, when both the time delay $\tau$ and the slope of the nonlinearity $f$ at $u=1$ are large enough, the travelling wave  provided by Theorem \ref{th:construction} does not converge to $1$ in $+\infty$, and oscillates around 1. In order to capture this behaviour, we need to strengthen Assumption \ref{ass:f} as follows.

\begin{assumption}[Bistable nonlinearity allowing oscillations]\label{ass:f-osc}
In addition to Assumption \ref{ass:f}, we suppose that the following properties hold true.
\begin{itemize}
\item[$(i)$] There exist two values $\alpha$ and $\beta$ such that
$\theta<\alpha<\beta<1$ and
\begin{equation*}
\begin{cases}
f'>0\text{ on $(0,\beta)$},\\
f'<0\text{ on $\left(\beta, M:=f(\beta)\right)$},\\
f(\alpha)=1\text{ and }f(M)>\alpha.
\end{cases}
\end{equation*}
\item[$(ii)$] $\displaystyle \int_0^{f(M)} \left(\min\{f(u),f(M)\}-u\right)du>0$.
\end{itemize}
\end{assumption}

Our result on oscillating travelling waves then reads as follows.

\begin{theorem}[Oscillations around 1 for large delay]\label{th:large-delay} Let Assumption \ref{ass:f-osc} hold. If $|f'(1)|$ is large enough then there exists $\tau_0>0$ such that, for all $\tau\geq \tau_0$, the travelling wave solution $\left(c,U\right)$ of~\eqref{eq} provided by Theorem \ref{th:construction} is such that the profile $U$ does not converge to $1$ as $x\to+\infty$, and oscillates around $1$.
\end{theorem}

Notice that, under an additional assumption, we can show that the above non convergence to~1 is actually a convergence to a wavetrain. This will be explored and precised in subsection \ref{ss:wavetrain}.

\medskip

The organization of the paper is as follows. Section \ref{s:construction} is devoted to the construction of a travelling wave thanks to a Leray-Schauder toplogical degree argument, that is we prove Theorem~\ref{th:construction}. In Section \ref{s:properties} we study two cases where the travelling wave converges to 1 in $+\infty$, proving Theorem~\ref{th:small-delay}. Last, in Section \ref{s:oscill}, we consider a case where oscillations occur, that is we prove Theorem~\ref{th:large-delay}.

\section{Construction of a travelling wave}\label{s:construction}

This section is devoted to the proof of Theorem \ref{th:construction}. The strategy is to first construct a solution in a box by a Leray-Schauder topological degree argument and then to let the box tend to the whole line, with enough estimates to guarantee a true bistable behaviour in the limit.

For $a>0$ and $0\leq \sigma \leq 1$, we consider the problem of
finding a speed $c=c_\sigma^a\in \R$ and a profile
$u=u_\sigma^a:[-a,a]\to \R$ such that
$$
P_\sigma(a)\quad\begin{cases}
\,-u''+cu^{\prime}=f(\bar u(\cdot -\sigma c\tau))-u\quad \text{ in }(-a,a)\vspace{5pt}\\
\,u(-a)=0, \quad  u(a)=1,\\
\end{cases}
$$
to which we shall sometimes add the normalization
\begin{equation}\label{normalization}
\max_{-a\leq x\leq 0}u(x)= \frac{\theta}{2},
\end{equation}
and where $\bar u$ denotes the extension of $u$ equal to $0$ on
$(-\infty,-a)$ and 1 on $(a,\infty)$ (in the sequel, for ease of
notation, we always write $u$ in place of $\bar u$). This realizes
a homotopy from a local problem ($\sigma =0$) to our nonlocal
delayed problem ($\sigma=1$) in the box $(-a,a)$. We shall
construct a solution to $P_1(a)$ by using a Leray-Schauder
topological degree argument. To make this scheme rigorous we will need several a priori bounds which are proved in the following subsection.

\subsection{A priori estimates}\label{ss:apriori}

Here we prove several lemmas on any profile $u$ and speed $c$ solving $P_\sigma (a)$ together with the normalization~\eqref{normalization}.

\begin{lemma}[A priori bounds for profiles]\label{lem:estim1}
For all $a >0$ and $0 \leq \sigma \leq 1$, any solution $(c_\sigma^a ,u_\sigma^a)=(c,u)$ of $P_\sigma (a)$ satisfies
$$
0\ <u < M \text{ on } (-a,a).
$$
\end{lemma}
\begin{proof}
If $u$ reaches its minimum at some point $x_0\in(-a,a)$ then the equation
yields $u(x_0)\geq f(u(x_0-\sigma c \tau))\geq 0$. In a similar fashion, if $u$ reaches its maximum at some point $x_1 \in (-a,a)$, then 
$u(x_1)\leq f(u(x_1-\sigma c \tau))$. On the one hand, if $u(x_1 - \sigma c \tau) >1$, then $f (u(x_1 - \sigma c \tau)) < u(x_1 - \sigma c \tau) \leq u(x_1)$, which is a contradiction. On the other hand, if $u(x_1 - \sigma c\tau) \leq 1$ then we conclude that $u (x_1) \leq \max_{[0,1]} f =M$. 

This proves that $0\leq u\leq M$
on $(-a,a)$. Note that $M$ is also the maximum of $f$ on the interval $[0,M]$. Thus
$$0\leq -u'' + c u' + u = f (u(\cdot - \sigma c \tau)) \leq M,$$
and the strong maximum principle yields  $0<u<M$ on $(-a,a)$.\end{proof}

In the next lemma, we prove some a priori monotonicity property. This will not only provide information on the shape of the travelling wave, but will also be quite useful in order to bound the speed.

\begin{lemma}[A priori monotony for a while]\label{lem:estim_mono}
Let $a >0$ and $0 \leq \sigma \leq 1$ be given, and $(c_\sigma^a ,u_\sigma^a)=(c,u)$ be a solution of $P_\sigma (a)$ with normalization 
\eqref{normalization}. If  $c \geq 0$ then $u$ is increasing on $[-a,\min(x_a + \sigma c \tau,a)]$, where 
\begin{equation}\label{def:xa}
x_a := \min \{x \in \R : u(x) = \theta \} >0.
\end{equation}
\end{lemma}

\begin{proof} Notice that the positivity of $x_a$  is an immediate consequence of the normalization \eqref{normalization}, so that we only need to deal with the monotonicity of the solution.

Assume first that $\sigma= 0$ or $c = 0$, so that 
$$
-u'' + cu' = f(u) - u\quad  \text{ in } (-a,a).
$$ 
Then the problem is local and one may proceed as in the seminal work of Berestycki and Nirenberg~\cite{Ber-Nir} to show that the solution is actually increasing on the whole interval $[-a,a]$. We omit the full details but, for the sake of clarity, briefly sketch the argument. First, arguing as in Lemma \ref{lem:estim1}, we see that $0<u<1$ on $(-a,a)$. One can therefore use a sliding method and define
$$
h^* := \min \{h \in [0,2a] \, : \ u(x+\zeta) \geq u(x), \forall x\in[-a,a-\zeta],\;\forall \zeta\in [h,2a] \} \in [0,2a).
$$
Assume by contradiction that $h^*>0$. Since $u(-a+h^*) > u(-a) = 0$ and $u(a)=1 > u(a-h^*)$, one can use the strong maximum principle to get that $u(\cdot + h^*) > u(\cdot)$, which contradicts the definition of $h^*$. Hence $h^*=0$ and $u$ is nondecreasing in $x$. The strict monotonicity is finally obtained by the strong maximum principle.

Now assume that $0<\sigma \leq 1$ and $c >0$. We first claim that $u$ is increasing on the interval $\left[-a,-a+\sigma c\tau \right]$. Indeed, if not, then there exists $x_0\in(-a,-a+\sigma c\tau)$ such that $u'(x_0)=0$ and $u''(x_0)\leq 0$. The travelling wave equation yields $u(x_0)\leq f(u(x_0-\sigma c\tau))=0$, as $x_0-\sigma c\tau \leq -a$ implies $u(x_0-\sigma c\tau)=0$. This contradicts $u>0$ and proves the claim. Hence, we can define
$$
\gamma ^*:=\sup\{\gamma>0:\, u \text{ is increasing on } [-a,-a+\gamma]\} \in [\sigma c\tau,2a ].
$$
Assume by contradiction that
$$
\gamma ^{*}<\min(a+x_a+ \sigma c\tau,2a).
$$
 Then $u'(-a+\gamma ^{*})=0$ and $u''(-a+\gamma ^{*})\leq 0$ so that the travelling wave equation yields $u(-a+\gamma  ^{*})\leq f(u(-a+\gamma ^{*}-\sigma c\tau))$. But $-a\leq -a+\gamma ^*-\sigma c\tau \leq x_a$ and the definition \eqref{def:xa} of $x_a$ implies that $0\leq u( -a+\gamma ^*-\sigma c\tau)\leq \theta$, hence $f(u(-a+\gamma ^{*}-\sigma c\tau))\leq u(-a+\gamma ^{*}-\sigma c\tau)$. As a result $u(-a+\gamma  ^{*})\leq u(-a+\gamma ^{*}-\sigma c\tau)$. Since $\sigma c\tau>0$, this contradicts the definition of $\gamma^{*}$. The lemma is proved.
\end{proof}

We are now in the position to prove a priori estimates on the speed~$c$ for which $P_\sigma (a)$ admits a solution with normalization~\eqref{normalization}.

\begin{lemma}[A priori bounds for speeds]\label{lem:apriori} $(i)$ There exist $c_{max} >0$ and $a_0>0$
such that, for all $a>a_0$, all $0\leq \sigma \leq 1$, any solution
$(c_\sigma ^a,u_\sigma ^a)=(c,u)$ of $P_\sigma(a)$ with normalization 
\eqref{normalization}  satisfies
$$
c \neq 0 \mbox{ and } c < c_{max}.
$$

 $(ii)$ Furthermore and up to enlarging $a_0$, for all $a>a_0$, any solution $(c_0^a,u_0 ^a)=(c,u)$ of the local problem $P_0(a)$ with normalization \eqref{normalization} satisfies
$$
0<c<c_{max}.
$$
\end{lemma}

\begin{rem}
Notice that the above lemma does not exclude the possibility that a solution with $c<0$ may exist for $0<\sigma\leq 1$. Nevertheless, excluding $c=0$ will be enough for our purpose. Indeed, we aim at constructing  a travelling wave with positive speed, and in order to apply a topological degree argument we only need to show that, along the homotopy $0\leq \sigma\leq 1$, no solution can escape through $c=0$. See subsection \ref{ss:degree} for more details, and \cite{Alf-Cov-Rao1} or \cite{Gri-Rao} for a similar trick.
\end{rem}

\begin{proof} 
Let us first exclude the case $c=0$ along $0\leq \sigma \leq 1$. Assume by contradiction that there are sequences $a\to \infty$, $0\leq \sigma \leq 1$, and $(c=0,u^a_\sigma)=(c=0,u)$ solving $P_\sigma (a)$ with normalization \eqref{normalization}. In other words, we are equipped with $(c=0,u)$ solving
$$
P(a)\quad \begin{cases}
\,-u''=f(u)-u\quad \text{ in }(-a,a)\vspace{5pt}\\
\,u(-a)=0, \quad u(a)=1,
\end{cases}
$$
and, from Lemma \ref{lem:estim_mono}, with the normalization $
u(0)=\frac \theta 2$. 
Problem $P(a)$ being local, we know (see the proof of Lemma \ref{lem:estim_mono}), that $u$ is actually increasing on the whole interval $[-a,a]$. Thus, letting $a\to \infty$ we end up with $(c=0,U)$ solving  
$$
\quad \begin{cases}
\,-U''=f(U)-U\quad \text{ in }\R\vspace{5pt}\\
\,U(-\infty)=0, \quad U(0)=\frac \theta 2, \quad U(+\infty)\in\{\theta,1\},\quad U'\geq 0.
\end{cases}
$$
Note that the limits $U(\pm \infty)$ above simply follow from the fact that these must be zeros of $f(u) -u$. Next, by a standart argument --- multiply the above equation by $U'$ and integrate over $\R$---  we end up with a contradiction --- since $\int _0 ^\theta (f(u)-u)du<0$ and $\int _0^1 (f(u)-u)du>0$ in view of Assumption~\ref{ass:f}.

\medskip

Let us now turn to the bound from above on $c$ along $0\leq \sigma \leq 1$. In view of Assumption~\ref{ass:f}, we can bound $f$ from above by a linear function. More precisely, there exists $K>0$ such that $f(u) \leq Ku$ for all $u \geq 0$. We now select $c_{max}>K$. It is then straightforward to see that
$$\overline{u} (x) := e^{x}$$
satisfies
$$-\overline{u}'' + c_{max} \overline{u}' > K \overline{u} (\cdot - \sigma c_{max} \tau) - \overline{u}.
$$

Let us now show that $c \leq c_{max}$, where $c$ is such that $P_\sigma (a)$ with normalization~\eqref{normalization} admits a solution. Proceed by contradiction and assume that $c > c_{max}$. Since $\overline{u} (+\infty) = +\infty$, one can select a large shift $X = X^a> > 1$ so that\begin{equation}
\label{shift}
\overline{u} (-a + X)>M.
\end{equation}
 Going back to the parabolic problem, we define
$$
\overline{v} (t,x) := \overline{u} (x+c_{max} t+ X), \quad (t,x)\in \R^{2},
$$
which satisfies
\begin{equation}
\label{eq-vplus}
\partial _t \overline{v} (t,x) > \partial _{xx} \overline{v} (t,x)+  K \overline{v} (t-\sigma \tau,x) - \overline{v}(t,x),\quad \text{ for } (t,x)\in \R ^{2}.
\end{equation}
We also define
$$
v(t,x):=u(x+ct),\quad (t,x)\in \R^{2},
$$
which satisfies
\begin{eqnarray}
\partial _t v(t,x)&=&\partial _{xx}v(t,x)+f(v(t-\sigma \tau,x))-v(t,x)\nonumber\\
& \leq &\partial _{xx}v(t,x)+ Kv(t-\sigma \tau,x) -v(t,x),\quad \text{ for } \vert x+ct\vert <a.\label{ineq-v}
\end{eqnarray}
Recall that $v$ is well-defined on the whole domain as $u$ is extended by constants outside of the interval $(-a,a)$.

We claim that
\begin{equation}
\label{claim1}
V(t,x):=v(t,x)- \overline{v} (t,x)< 0, \quad -\sigma\tau \leq t\leq 0,\; x\in \R.
\end{equation}
Indeed if $x+ct<-a$ then $v(t,x)=0$ so the inequality is clear. On the other hand 
if $x+ct\geq -a$ then the monotony of $\overline{u}$ (notice that $(c_{max} -c)t\geq 0$ when $-\sigma \tau \leq t\leq 0$) and~\eqref{shift} imply $\overline{v}(t,x)>M$, which proves \eqref{claim1}.

Next, since $c>c_{max}$, we can define the first touching time, namely
$$
t_0:=\inf \{t>0 \, : \  \exists x\in \R, \, V(t,x)>0\}\geq  0.
$$
In particular $V(t_0,\cdot)\leq 0$ on $\R$, and there is $x_0\in \R$ such that $V(t_0,x_0)=0$. Clearly the touching point $(t_0,x_0)$ is such that $x_0+ct_0>-a$. If $(t_0,x_0)$ is such that $\vert x_0+ct_0\vert <a$ (i.e. in the region where $v$ can be differentiated) then $\partial _t V(t_0,x_0)\geq 0$ and $\partial _{xx}V(t_0,x_0)\leq 0$. Then, subtracting \eqref{eq-vplus} from \eqref{ineq-v} and evaluating at point $(t_0,x_0)$ leads to a contradiction. If $(t_0,x_0)$ is such that $x_0+ct_0>a$ (i.e. in the region where $v\equiv 1$) then $\partial _t V(t_0,x_0)=-\partial _t \overline{v} (t_0,x_0)<0$ which contradicts the definition of $t_0$. It remains to exclude the case $x_0+ct_0=a$. In this case we have $1= v(t_0,x_0)= \overline{v} (t_0,x_0)= \overline{u} (x_0+c_{max} t_0+X)$ so that 
\begin{equation}
\label{moinsb}
\overline{u} (z) \leq \frac \theta 4, \quad \forall z\leq x_0+c_{max} + X + \ln \left(\frac \theta 4 \right).
\end{equation}
But $v(t_0,x)\leq \overline{v} (t_0,x)$ may be rewritten as $u(z)\leq \overline{u} (z+(c_{max}-c)t_0+X)= \overline{u} (z+x_0+c_{max} t_0+X - a)$. For any $a > - \ln \left( \frac \theta 4 \right)$, we deduce from \eqref{moinsb} that
$$
u(z)\leq \frac\theta 4, \quad \forall z\leq 0.
$$
This contradicts the normalization \eqref{normalization}, and thus we have $c\leq c_{max}$. This concludes the proof of item $(i)$ of the lemma.

\medskip

Last, we turn to the bound from below in $(ii)$, which is only concerned with the local case $\sigma =0$. We  select $\delta>0$ small enough and $g$ a bistable type function with $-\delta$, $\theta +\delta$ and $1-\delta$ its fixed points, such that $g<f$ and 
$$
\int _{-\delta}^{1-\delta} (g(u)-u)du>0.
$$ 
It is well-known that there exists a bistable travelling wave $(c_\delta>0,u_\delta)$ for the nonlinearity $g$, i.e. solving
$$- u_\delta '' + c_\delta u_\delta ' = g (u_\delta) - u_\delta,$$
with $u_\delta (-\infty) =-\delta < u_\delta (\cdot) < u_\delta (+\infty) = 1-\delta$. Moreover, $u_\delta$ is an increasing function of $x \in \mathbb{R}$. Reproducing the above proof --- that is using a shifted $u_\delta$ as a subsolution of the parabolic problem--- we find that $0<c_\delta\leq c$ for all $a$ large enough, which concludes the proof of $(ii)$ and thus of the lemma.
\end{proof}

Next we prove some technical estimates on the behaviour of $u$ on the left of $x_a$. We recall that, as defined in~\eqref{def:xa}, $x_a$ is the leftmost (and, as it will turn out in Lemma \ref{lem:theta}, unique) point where $u$ takes the value $\theta$.

\begin{lemma}[A priori control on the left of $x_a$ and of the slope at $x_a$]\label{lem:apriorileft2}  For all $a>0$, all $0\leq \sigma \leq 1$, any solution
$(c_\sigma ^a,u_\sigma ^a)=(c,u)$ of $P_\sigma(a)$ with $c\geq 0$ and normalization \eqref{normalization}
satisfies 
\begin{eqnarray}\label{controlleft2}
\psi(x)&:=&\theta\frac{e^{\frac{c+\sqrt{c^2+4}}{2}x}e^{\sqrt{c^2+4}(a+\frac{x_a}{2})}-e^{\frac{c-\sqrt{c^2+4}}{2}x}e^{\sqrt{c^2+4}\frac{x_a}{2}}}{e^{\frac c 2 x_a}(e^{\sqrt{c^{2}+4}(a+x_a)}-1)}\nonumber \vspace{5pt}\\
&& \leq u(x)\leq \theta e^{c(x-x_a)}=:\varphi(x),
\end{eqnarray}
for all  $x\in [-a,x_a]$. This in turn implies the following: there is $s_2>0$ such that for all $a>0$, all $0\leq \sigma \leq 1$, any solution
$(c_\sigma ^a,u_\sigma ^a)=(c,u)$ of $P_\sigma(a)$ with $c\geq 0$ and normalization \eqref{normalization}
satisfies
\begin{equation}\label{controlslope2}
s_2\geq u'(x_a)\geq \theta c \geq 0.
\end{equation}
\end{lemma}

\begin{proof} First, $u$ satisfies $-u''+cu'+u\geq 0$ on $(-a,x_a)$ with the boundary conditions $u(-a)=0$, $u(x_a)=\theta$. Since 
$\psi$ satisfies $-\psi ''+c\psi '+\psi =0$ on $(-a,x_a)$ and $\psi(-a)=0$, $\psi(x_a)=\theta$, the estimate from below in \eqref{controlleft2} follows from the comparison principle.

{}From $c\geq 0$ and the monotonicity Lemma~\ref{lem:estim_mono}, we know that $u(x-c\sigma \tau)\leq u(x)\leq \theta=u(x_a)$ for all $x\in [-a,x_a]$, which in turn implies $f(u(x-c\sigma\tau))-u(x)\leq u(x-c\sigma \tau)-u(x)\leq 0$. Hence $u$ satisfies $-u''+cu'\leq 0$ on $(-a,x_a)$ with the boundary conditions $u(-a)=0$, $u(x_a)=\theta$, so that the estimate from above in \eqref{controlleft2} follows from the comparison principle.

Last, \eqref{controlleft2} enforces $\psi'(x_a)\geq u'(x_a)\geq \varphi'(x_a)$. Since
$$
\psi'(x_a)= \theta  \left(\frac c 2+\frac{\sqrt{c^2+4}}{2}\frac{e^{(a+x_a)\sqrt{c^2+4}}+1}{e^{(a+x_a)\sqrt{c^2+4}}-1}\right)\leq \theta  \frac {c_{max}+\sqrt{c_{max}^2+4}\sup_{X\geq 2 a_0} \frac{e^{X}+1}{e^{X}-1}}{2}=:s_2,
$$
and $\varphi'(x_a)=\theta c$, estimate \eqref{controlslope2} is proved.
\end{proof}

\subsection{Construction in the box}\label{ss:degree}

Equipped with the above {\it a priori} estimates, we now use a
Leray-Schauder topological degree argument (for related arguments see e.g.
\cite{Ber-Nic-Sch}, \cite{Ber-Nad-Per-Ryz} or \cite{Alf-Cov-Rao1}  in a KPP context, \cite{Alf-Cov-Rao2} in a bistable context) to construct a solution $(c,u)$ to $P_1(a)$ with normalization \eqref{normalization}.

\begin{proposition}[A solution in the
box]\label{prop:sol-boite} There exist $C_0 >0$ and $a_0>0$ such
that, for all $a\geq a_0$, problem $P_1(a)$ admits a solution
$(c^a,u^a)$ with normalization \eqref{normalization},
which is such that
$\Vert u^a \Vert _{C^2(-a,a)}\leq C_0$. 
\end{proposition}

\begin{proof} For a given $c\in \R$ and a given nonnegative function $v$ defined on $(-a,a)$, consider the
family $0\leq\sigma\leq1$ of linear problems
\begin{equation}\label{droite-gelee}
P_\sigma^c(a) \;\begin{cases}\, -u''+cu'=f(\bar v(\cdot -\sigma c\tau))-v
\quad &\text{ in } (-a,a)\vspace{3pt} \\
u(-a)=0,\quad u(a)=1,
\end{cases}
\end{equation}
where, as before, $\bar v$ denotes the extension of $v$ by 0 on $(-\infty,-a]$ and by 1 on $[a,\infty)$. Denote by $\mathcal K _\sigma$ the mapping of the Banach space
$X:=\R \times \{v\in C^{1,\alpha}([-a,a]): v(-a)=0\}$, equipped with the norm $\Vert
(c,v)\Vert_X:=\max\left(|c|,\Vert v\Vert
_{C^{1,\alpha}}\right)$, onto itself defined by
$$
\mathcal K _\sigma:(c,v)\mapsto \left( \frac{\theta}{2} -\max_{-a\leq x\leq 0}v(x)+c,u_\sigma
^c:=\text{ the solution of } P_\sigma ^c(a)\right).
$$
Constructing a solution $(c,u)$ of $P_1(a)$ with normalization \eqref{normalization} is equivalent to
showing that the kernel of $\text{Id} -\mathcal K _1$ is
nontrivial. The operator $\mathcal K _\sigma$ is compact and depends
continuously on the parameter $0\leq \sigma \leq 1$. Thus the
Leray-Schauder topological argument can be applied. 

 From Lemma \ref{lem:estim1} and the interior elliptic estimates there is $C_M>0$ such that, for any $0\leq \sigma \leq 1$, any solution $(c,u)$ of $P_\sigma(a)$ with $0\leq c\leq c_{max}$ has to satisfy $\Vert u\Vert_{C^{1,\alpha}}<C_M$. Let us define the
open set
$$
S:=\left\{(c,v):\, 0 < c< c_{max},\;v>0 \text{ on } (-a,a],\; v'(-a)>0,\;
\Vert v \Vert _{C^{1,\alpha}}< C_M\right\}\subset X.
$$
It follows from the above
and Lemma~\ref{lem:apriori} $(i)$
 that, for any $a\geq a_0$, any $0\leq \sigma \leq 1$, any fixed point of $\mathcal K _\sigma$ satisfies $u>0$ on $(-a,a]$, $u'(-a)>0$ by the Hopf lemma, $c\neq 0$ and $c<c_{max}$, and therefore the operator
$\text{Id} -\mathcal K _\sigma$ cannot vanish on the boundary
$\partial S$. By the homotopy invariance of the degree we thus
have $\text{deg}(\text{Id}-\mathcal
K_1,S,0)=\text{deg}(\text{Id}-\mathcal K_0,S,0)$. Now, as far as $\mathcal K _0$ is concerned, we know from Lemma \ref{lem:apriori} $(ii)$ that $c\leq 0$ is impossible for a fixed point. We can therefore enlarge the speed interval in $S$ without changing the degree, that is $\text{deg}(\text{Id}-\mathcal K_0,S,0)=\text{deg}(\text{Id}-\mathcal K_0,S(a),0)$, where
$$
S(a):=\left\{(c,v):\, -c_{min}(a) < c< c_{max},\;v>0,\;
\Vert v \Vert _{C^{1,\alpha}}< C_M\right\}\subset X,
$$
and $c_{min}(a)\geq 0$ is to be selected below.

Roughly speaking, the role of \eqref{droite-gelee} was to get rid of the nonlocal term. Now, in order to get rid of the nonlinearity, we consider the family $0\leq \sigma\leq 1$ of local problems
\begin{equation}\label{droite-gelee2}
\widetilde P_\sigma^c(a) \;\begin{cases}\, -u''+cu'=\sigma(f(v)-v)
\quad &\text{ in } (-a,a)\vspace{3pt} \\
u(-a)=0,\quad u(a)=1,
\end{cases}
\end{equation}
and let $\widetilde{\mathcal K}_\sigma$ be the associated solution operator, namely
$$
\widetilde{\mathcal K} _\sigma:(c,v)\mapsto \left(\frac \theta 2-\max_{-a\leq x\leq 0}v(x)+c,\widetilde u_\sigma
^c:=\text{ the solution of } \widetilde P_\sigma ^c(a)\right).
$$
It is clear that, for any $0\leq\sigma \leq 1$, a fixed point $(c,u)$ of $\widetilde{\mathcal K}_\sigma$ satisfies $0<u<M$ in $(-a,a)$ (reproduce Lemma \ref{lem:estim1}). Furthermore, one can also proceed as in the proof of Lemma~\ref{lem:apriori} to check that the inequality $c< c_{max}$ still holds. Indeed, the same argument applies noting that the exponential function $e^{x + c_{max} t}$ is a supersolution of the parabolic equation
$$\partial_t u = \partial_{xx} u + \sigma \left( f(u)-u \right),$$
associated with $\widetilde P_\sigma^c (a)$, provided as before that $c_{max}>K$ where $f(u) \leq K u$ for all $u \geq0$. Also, any solution of $\widetilde P_\sigma^c (a)$ satisfies $-u''+cu'\geq -\sigma u\geq -u$ so that, by the comparison principle, $u\geq w$ where $w$ solves
$$
-w''+cw'+w=0 \quad \text{ on } (-a,a), \quad w(-a)=0, \quad w(a)=1.
$$ 
After explicitly computing $w$ we see that, for any $a\geq a_0$, $w(0)\to 1$ as $c\to -\infty$ so that there is $c_{min}(a)\geq 0$ such that, in order not to miss the normalization, one needs $-c_{min}(a)<c$. Then the operator 
$\text{Id} -\widetilde{\mathcal K} _\sigma$ cannot vanish on the boundary
$\partial S(a)$. Hence, by the homotopy invariance of the degree and the fact that $\mathcal K _0=\widetilde {\mathcal K }_1$, we have $\text{deg}(\text{Id}-\mathcal K_0,S(a),0)=\text{deg}(\text{Id}-\widetilde{\mathcal K}_0,S(a),0)$.

Before  computing
$\text{deg}(\text{Id}-{\widetilde {\mathcal K}}_0,S(a),0)$
by using two additional homotopies, let us observe that the solution of $\widetilde P _0 ^c(a)$  is given by
\begin{equation}\label{tauzero}
\widetilde u_0^c(x)=\displaystyle \frac{e^{cx}-e^{-ca}}{e^{ca}-e^{-ca}}
\quad\text{ if } c\neq 0,\quad \widetilde u_0^c(x)=\frac {x+a}{2a}
\quad\text{ if } c=0.
\end{equation}
In particular, $\widetilde u_0^{c}(0)$ is decreasing with respect to $c$ and, there is a unique $c_0$ such that
$\widetilde u_0^{c_0}(0)=\frac \theta 2$. As a result, the operator $\widetilde {\mathcal K} _0$ has a unique fixed point $(c_0,\widetilde u_{c_0})$ which moreover belongs to $S(a)$.

Let us now perform two additional homotopies. First, consider, for $0\leq
\sigma \leq 1$,
$$
\mathcal G _\sigma:(c,v)\mapsto \left(\frac \theta 2-(1-\sigma)\max_{-a\leq x\leq 0}v(x)-\sigma \widetilde u_0^c(0)+c,\widetilde u_0^c:=\text{ the solution of } \widetilde P_0 ^c(a)\right).
$$
Again, $\text{Id} -\mathcal G _\sigma$ does not vanish on the
boundary $\partial  S(a)$. By the homotopy invariance of the degree and the fact that $\widetilde{ \mathcal K} _0= {\mathcal G }_0$, we have $\text{deg}(\text{Id}-\widetilde{\mathcal K}_0,S(a),0)=\text{deg}(\text{Id}-\mathcal G _1,S(a),0)$. Then, consider, for $0\leq \sigma \leq
1$,
$$
\mathcal H _\sigma:(c,v)\mapsto \left(\frac \theta 2- \widetilde u_0^c(0)+c,\sigma
\widetilde u_0^c+(1-\sigma) \widetilde u _{c_0}\right),
$$
where $(c_0,\widetilde u_{c_0})$ was defined in the previous paragraph. If $\mathcal H _\sigma (c,v)=(c,v)$ for some $(c,v)\in\partial
 S(a)$, then it follows from the previous paragraph  that $(c,\widetilde u_0^{c})\equiv (c_0,\widetilde u_{c_0})$ and therefore $(c,v)\equiv (c_0,\widetilde u_{c_0})$  so that
$(c,v)\in
\partial  S(a)$ solves the local problem $ \widetilde P_0^{c}(a)$ and satisfies normalization \eqref{normalization}, which cannot be. Therefore $\text{Id} -\mathcal H
_\sigma$ does not vanish on the boundary $\partial  S(a)$. Since
$\mathcal H _1=\mathcal G _1$ we have
$\text{deg}(\text{Id}-\mathcal G _1,
S(a),0)=\text{deg}(\text{Id}-\mathcal H_0 ,S(a),0)$, where
$$
\text{Id}-\mathcal H _0:(c,v)\mapsto \left(
\widetilde u_0^c(0)-\frac \theta 2,v-\widetilde u _{c_0}\right).
$$
As seen above, $\widetilde u_0^c(0)$ is strictly decreasing in $c$ so the degree of the
first component of the above operator is $-1$. Clearly the degree
of the second one is 1. Hence $\text{deg}(\text{Id}-\mathcal
H_0,S(a),0)=-1$
 so that $\text{deg}(\text{Id}-\mathcal
K_1,S,0)=-1$ and there is a solution
 $(c^a,u^a)\in S$ of $P_1(a)$ with normalization~\eqref{normalization}. Together with standard estimates, this concludes the proof of
the proposition. \end{proof}

In the sequel, in order to complete the construction of a bistable travelling wave as stated in Theorem \ref{th:construction}, we actually need to  strengthen the above result. More precisely, we show below the existence of a solution that crosses the value $\theta$ only once.

\begin{proposition}[A solution such that $\theta$ is attained only at $x=x_a$]\label{lem:theta}
 Let $C_0 >0$ and $a_0>0$ be as in Proposition \ref{prop:sol-boite}. Then, for all $a\geq a_0$, problem $P_1(a)$  admits a solution
$(c^a,u^a)$ with normalization \eqref{normalization},
which is such that
$\Vert u^a \Vert _{C^2(-a,a)}\leq C_0$ and 
$$
u^a (x)=\theta \; \text{ if and only if }\;  x=x_a,
$$
where $x_a>0$ was defined in \eqref{def:xa}
\end{proposition}

\begin{proof} Let $a\geq a_0$ be given. We again work along the homotopy $0\leq \sigma\leq 1$. From the proof of Proposition~\ref{prop:sol-boite}, we know that there is a
solution $(c_\sigma,u_\sigma)$ of
\begin{equation}\label{twsigma}
\left\{\begin{array}{l}
-u_\sigma''+c_\sigma u_\sigma'=f(u_\sigma(\cdot-\sigma c_\sigma \tau))-u_\sigma \quad \text{ in }(-a,a)\vspace{3pt}\\
 u_\sigma(-a)=0,\quad u_\sigma(0)=\frac \theta 2,\quad  u_\sigma(a)=1.
  \end{array}
\right.
\end{equation}
For the local case $\sigma =0$, we also know from the proof of Lemma~\ref{lem:estim_mono} that the (unique) solution
$u_0$ is increasing and therefore satisfies $u_0(x)=\theta$ if and only if $x=x_a$. 

In particular, using also~\eqref{controlslope2} and the fact that $c_0 >0$, the function $u_0$ belongs to the open set
$$S_\theta := S \cap \{ (c,v) \, : \ v (x_1) = v(x_2) = \theta \Rightarrow (v ' (x_1) > 0 \mbox{ and } x_1 = x_2) \} \subset X,$$
where $S$ and $X$ were defined in the proof of Proposition~\ref{prop:sol-boite}. Let us prove that $P_\sigma (a)$ together with \eqref{normalization} does not admit a solution on the boundary of $S_\theta$ (with respect to the topology of the Banach space $X$ as before). As we already dealt with the boundary of $S$, we only need to show here that, for any $\sigma \in (0,1]$, there is no solution $u_\sigma$ such that $u_\sigma \geq \theta$ on $[x_a,a]$, and either 
\begin{equation}\label{exclude1}
u_\sigma ' (x_a) = 0
\end{equation}
or there exists $x^* > x_a$ with 
\begin{equation}\label{exclude2}
u_\sigma (x^*) = \theta .
\end{equation}
Using again \eqref{controlslope2}, we have $u_\sigma '(x_a)\geq \theta c_\sigma >0$ and we may immediately rule out the former case \eqref{exclude1}. Now proceed by contradiction and assume that $u_\sigma (x^*) = \theta$ for some $x^* > x_a$. Without loss of generality, we also assume that $u_\sigma > \theta$ in the open interval $(x_a, x^*)$. Clearly $u'_{\sigma}(x^*)=0$ and $u''_{\sigma }(x^*)\geq 0$. Testing the equation at point $x^*$ we get
$$
f(u_{\sigma }(x^*-\sigma c_{\sigma }\tau))\leq u_{\sigma }(x^*)=\theta.
$$
Moreover, by the monotonicity Lemma~\ref{lem:estim_mono} and the positivity of the speed, we have that $x^*-\sigma c_{\sigma }\tau\in(x_a,x^{*})$. Thus $u_{\sigma}(x^*-\sigma c_{\sigma }\tau)>\theta$ so that 
$f(u_{\sigma^{*}}(x^*-\sigma ^*c_{\sigma ^*}\tau))>\theta$, which is a contradiction ruling out \eqref{exclude2}.

It is now straightforward to show that the topological degree argument of the previous proposition can be performed in the set $S_\theta$ instead of $S$. This completes the proof of the proposition.\end{proof}

\subsection{Construction of a travelling wave}\label{ss:tw}

Equipped with the solution $(c^{a},u^{a})$ of $P_1(a)$ of Proposition
\ref{lem:theta}, we now let $a\to \infty$. This enables to
construct, up to extraction of a subsequence $a_n \to \infty$, a speed
$0 \leq c \leq c_{max} $ and a function $U:\R\to[0,M]$ in
$C^2 _b (\R)$ such that
\begin{equation}\label{eq-onde-construite}
-U''+c U'=f(U(\cdot -c\tau))-U \quad \text{ in }\R,
\end{equation}
and
\begin{equation}\label{masse-onde-construite}
U(0)= \frac{\theta}{2}.
\end{equation}
 From the strong maximum principle (recall that $M= \max_{[0,1] } f = \max_{[0,M]} f$)  we immediately deduce that $0<U<M$. Moreover, it follows from
Lemma \ref{lem:estim_mono} that $U$ is nondecreasing at least on the interval $(-\infty,c\tau]$. In particular, the limit $U(-\infty)< \frac{\theta}{2}$ exists and, since it must satisfy $f(U(-\infty)) - U(-\infty) =0$, one infers that 
\begin{equation}\label{left-limit}
U(-\infty)=0.
\end{equation}
To complete the construction of a travelling wave as stated in Theorem~\ref{th:construction}, our main tasks in this subsection are to prove that $c>0$ and $\liminf_{x\to+\infty} U(x)>\theta$. To do so, we need to go back to the problem in the box to get further estimates. Notice that the proofs of these new estimates involve passing to the limit as $a \to +\infty$, which is why they were not dealt with in subsection~\ref{ss:apriori}.

\begin{lemma}[Boundedness of the first $\theta$ point]
Recall that $x_a >0$, defined in \eqref{def:xa}, is such that $u^a (x_a) = \theta$ and $u^a (x)< \theta$ for all $x \leq x_a$. Then
$$\limsup_{a \to \infty
} x_a < +\infty.$$
\end{lemma}
\begin{proof}
We proceed by contradiction and assume that, up to a subsequence, $x_a \to +\infty$. In particular, we can assume without loss of generality that $U (x) \leq \theta$, and even $U(x)< \theta$ for all $x \in \mathbb{R}$ from the strong maximum principle. Moreover, the function $U$ is nondecreasing on $\R$ in virtue of Lemma \ref{lem:estim_mono}. As a result
$$-U''\leq -U'' + cU' =f(U(\cdot-c\tau))-U\leq U(\cdot- c \tau) -U \leq 0 \quad \text{ on } \R.
$$
The function $U$ is convex and bounded on $\R$ so it has to be constant and equal to $\frac{\theta}{2}$, which clearly cannot be.
\end{proof}

Let us now go back to showing that $U$ is a travelling wave. To that aim we first show that $ c >0$.

\begin{proof}[Proof of $c>0$ for the constructed wave]
Recalling that $(c^a,u^a)$ is provided by Proposition \ref{lem:theta}, from the above lemma, we can now take a subsequence $a \to +\infty$ (still denoted by $a$ for simplicity) such that $x_a \to x_\infty\geq 0$, and we are equipped with $U$ solving \eqref{eq-onde-construite}, \eqref{masse-onde-construite}, \eqref{left-limit}, nondecreasing on $(-\infty,x_\infty+c\tau]$ but also satisfying
\begin{equation}\label{theta-droite}
U(x_\infty) = \theta, \quad U\geq \theta \quad \text{ on } [x_\infty,\infty).
\end{equation}
Assume by contradiction that $c=0$, so that $U$ solves
$$-U'' = f(U) - U\quad \text{ on } \R .$$
By performing a standard phase plane analysis and because $U(-\infty)=0$, we see that such a positive and bounded solution either tends to 1 in $+\infty$ or oscillates around $\theta$ after $x_\infty$. In the former case, we multiply the equation by $U'$, integrate over $\R$ and get $0=\int _0^1 (f(u)-u)du$ which contradicts \eqref{ass:integrale}, whereas the latter case is excluded by \eqref{theta-droite}.  Hence $c>0$ and, as a consequence, there is $c_{min}>0$ such that
\begin{equation}\label{c-et-cmin}
c^a\geq c_{min} >0,
\end{equation}
 for all $a$ large enough.
\end{proof}


Notice that from \eqref{c-et-cmin} and \eqref{controlslope2} we get
\begin{equation}\label{pente-onde-construite}
U'(x_\infty)\geq \theta  c_{min}>0.
\end{equation} 
This now allows us to improve Lemma \ref{lem:estim_mono} for the solution in the box. 

\begin{lemma}[A priori monotony, a bit further]\label{lem:further}
Up to increasing $a_0$, there is $\zeta >0$ such that, for all $a\geq a_0$, the solution $(c^a,u^a)$ of Proposition~\ref{lem:theta} satisfies $u(x_a+ c\tau)\geq \theta +\zeta$.

Furthermore, there is $\eta>0$ such  that, for all $a\geq a_0$, the solution $(c^a,u^a)$ of Proposition~\ref{lem:theta} is increasing on $[-a,x_a + c\tau +\eta]$.
\end{lemma}

\begin{proof} Assume by contradiction that the first conclusion is false. Then
there are a sequence of box sizes $a_n\to \infty$, and a sequence of solutions $(c_n,u_n)$ on $(-a_n,a_n)$ such that $u_n(x_{a_n} + c_n\tau)\leq \theta +\frac 1n$. As a result, $\theta \leq u_n\leq \theta + \frac 1 n$ on $[x_{a_n},x_{a_n} + c_{min}\tau]$. Letting $n\to \infty$ we can again construct a solution $\tilde U$ satisfying \eqref{eq-onde-construite}, \eqref{masse-onde-construite} as well as \eqref{pente-onde-construite}. This last estimate is in contradiction with the fact that $\tilde U\equiv \theta$ on $[x_\infty,x_\infty+c_{min}\tau]$.

Next assume by contradiction that the second conclusion is false. Then
there are a sequence of box sizes $a_n\to \infty$, a sequence of solutions $(c_n,u_n)$ on $(-a_n,a_n)$, and a sequence of points $y_n \in [x_{a_n}+ c_n\tau , x_{a_n}+ c_n\tau+\frac 1n ]$ where $u'_n(y_n)=0$ and $u''_n(y_n)\leq 0$. Testing the equation at point $y_n$ yields
$$
f(u_n(y_n-c_n\tau))\geq u_n(y_n).$$
As we can choose without loss of generality $y_n$ such that $u_n$ is increasing on $[x_{a_n}, y_n]$, we get
$$f (u_n (y_n - c_n \tau) \geq u_n(x_{a_n}+ c_n\tau).
$$
{}From the first part of the lemma, we infer that
$$ f(u_n (y_n - c_n \tau)) \geq \theta + \zeta.$$
Now from the mean value theorem, we get the existence of $\rho_n\to 0$ such that
\begin{eqnarray*}
f(u_n (y_n - c_n \tau))&=& f(u_n(x_{a_n}))+(y_n-x_{a_n} - c_n\tau)f'(u_n(x_{a_n} +\rho _n))u'_n(x_{a_n} + \rho _n) \\
& =& \theta+(y_n-x_{a_n} - c_n\tau)f'(u_n(x_{a_n} + \rho _n))u'_n(x_{a_n}+ \rho _n)\\
& \geq & \theta +\zeta.
\end{eqnarray*}
We claim (see below) that $\sup _{n\in \N} \Vert u'_n ( x_{a_n}+\cdot)\Vert _{L^{\infty}(-1,1)}<\infty$. Then lettting $n\to\infty$ in the above inequality gives a contradiction.

The claim follows from usual estimates. Recall that $
-u''_n+c_n u_n'=f(u_n(\cdot -c_n\tau))-u_n$.
Since both $c_n$ and the  $L^\infty$ norm of the right hand side member are
uniformly bounded with respect to $n$, the interior elliptic
estimates \cite[Theorem 9.11]{Gil-Tru} imply that, for all $R>0$, all $p>1$, the
sequence $(u_n (x_{a_n}+\cdot))$ is bounded in $W^{2,p}(R,R)$ and therefore
in $C^{1,\beta}[-R,R]$, $\beta:=1-\frac 1p$, from Sobolev
embedding theorem.
 \end{proof}

Recalling from \eqref{controlslope2} and~\eqref{c-et-cmin} that $(u^a) '(x_a)\geq s_1:=\theta c_{min}>0$, we get from standard elliptic estimates as in the previous paragraph that the solution $(c^a,u^a)$ of Proposition~\ref{lem:theta} satisfies
\begin{equation}
\label{pente-pres-zero}
(u^a) '(x)\geq \frac 12 s_1>0, \quad \forall x_a \leq x \leq x_a + \delta ,
\end{equation}
for some $\delta >0$. Let us now fix
$$
0<\delta ^*<\min\left(\frac{\delta s_1}{2}, \frac {\eta s_1}2, \frac{c_{min}\tau s_1}{2}\right),
$$
where $\eta >0$ is as in Lemma \ref{lem:further}. Notice that $\delta ^*>0$ is independent of $a\geq a_0$. Also, in view of Assumption \ref{ass:f} and up to reducing $\delta ^*>0$ if necessary, we can assume that
\begin{equation}
\label{pastropbas}
f(u)\geq \theta +\delta ^*,\quad \forall 1 \leq u\leq M.
\end{equation}
In view of \eqref{pente-pres-zero}, for all $a\geq a_0$ we have that $u^a \left(x_a + \frac{2\delta ^*}{s_1}\right)\geq \theta +\delta ^*$. As a result, for all $a\geq a_0$, the quantity
$$
b=b(a):=\sup\{x\geq 0:\, u^a (y)<\theta+\delta ^*,\  \forall 0\leq y<x_a +x\}
$$
is bounded from above by $\frac{2\delta ^*}{s_1}$. Notice also that since $\frac{2\delta ^{*}}{s_1}<c_{min}\tau$ we know from Lemma \ref{lem:further} that $u$ is increasing on $[-a,x_a + b(a)]$. 

We are now in the position to prove that
\begin{equation}
\label{liminf-boite}
u^a (x)\geq \theta +\delta ^*,  \quad \forall x_a + b(a)\leq x\leq a.
\end{equation}
Assume \eqref{liminf-boite} is false. Then from Lemma \ref{lem:further} and the boundary condition $u^a (a)=1$, there must be a point $x_{min}\geq x_a+ c\tau +\eta$ where $(u^a) '(x_{min})=0$, $(u^a) ''(x_{min})\geq 0$ and
$u^a (x_{min})=\min_{x_a+ b(a)\leq x\leq x_{min}}u^a(x)<\theta +\delta ^{*}$. Testing the equation at point $x_{min}$ we see that
\begin{equation}\label{test-xmin}
f(u^a(x_{min}-c\tau))\leq u^a(x_{min})<\theta +\delta ^{*}.
\end{equation}
But $x_{min}-x_a - c\tau \geq \eta \geq \frac{2\delta ^{*}}{s_1}\geq b(a)$ so that
$u^a (x_{min})\leq u^a (x_{min}-c\tau)\leq M$. On the one hand if $\theta<u^a (x_{min})\leq u^a (x_{min}-c\tau)<1$ then $f(u^a (x_{min}-c\tau))>u^a (x_{min}-c\tau)\geq u^a (x_{min})$, which contradicts \eqref{test-xmin}. On the other hand, if $1\leq u^a (x_{min}-c\tau)\leq M$ it follows from \eqref{pastropbas} that $f(u^a (x_{min}-c\tau))\geq \theta +\delta ^*$, which again contradicts \eqref{test-xmin}. We have thus proved \eqref{liminf-boite}.

\begin{proof}[End of proof of Theorem \ref{th:construction}] Going back to the solution $U$ on the whole line $\R$ constructed in the beginning of subsection \ref{ss:tw}, passing to the limit in \eqref{liminf-boite} as $a \to +\infty$ implies that
$$
\liminf _{x\to +\infty}U(x)\geq \theta +\delta ^*>\theta.
$$
Therefore, the function $U$ is a travelling wave and, up to the shift $U (\cdot +x_\infty)$, it only remains to show that $U$ is increasing on $(-\infty,c\tau]$ to validate all
 the properties stated in Theorem \ref{th:construction}. We actually prove
 \begin{equation}\label{strict-croissance}
 U'(x)>0, \; \forall x< c\tau.
 \end{equation}
 Assume by contradiction that $U'(x_0)=0$ for some $x_0< c\tau$. Since $U$ is nondecreasing on $(-\infty,c\tau]$, this enforces $U''(x_0)=0$ and thus, from the equation, $U(x_0)= f(U(x_0-c\tau))\leq U(x_0-c\tau)$ and hence $U\equiv U(x_0)$ on $[x_0-c\tau,x_0]$ and thus on $(-\infty,x_0]$, a contradiction. 
\end{proof}

\section{Convergence to 1}\label{s:properties}

In this section we investigate further the behaviour of the travelling wave behind the front, i.e. as $x \to +\infty$. 
Theorem~\ref{th:small-delay} is proved in two subsections, each dealing with one of the two sufficient conditions --- \eqref{add-bound} and \eqref{small-delay}--- for convergence to 1. 

Before going further, observe that if $(c,U)\in \R\times C^2_b(\R)$ is a bistable travelling wave then $c\neq 0$. Indeed, assume by contradiction that $c=0$, so that $U$ solves
$-U'' = f(U) - U$  on  $\R$. By performing a standard phase plane analysis and because $U(-\infty)=0$, $\liminf_{x\to +\infty}U(x)>\theta$ , we see that such a positive and bounded solution has to tend to 1 in $+\infty$. We multiply the equation by $U'$, integrate over $\R$ and get $0=\int _0^1 (f(u)-u)du$ which contradicts \eqref{ass:integrale}. 

\subsection{Under an additional bound on $f$}\label{ss:add_bd}

We prove the following proposition, from which Theorem~\ref{th:small-delay} under assumption \eqref{add-bound} immediately follows.
\begin{proposition}[Convergence to 1 when $M=1$]
Assume that
$$
M := \max_{[0,1]} f =1.
$$
Then any bistable travelling wave $(c,U)\in \R \times C^{2}_b(\R)$   satisfies $0< U<1$ and $U(+\infty)=1$.
\end{proposition}
\begin{proof}According to Theorem~\ref{th:construction}, the travelling wave we have constructed in the previous section satisfies $0<U<1$. Let us check that this remains true for any bistable travelling wave $(c,U)\in \R \times C^2_b(\R)$. Recall first that the inequality $U>0$ is part of Definition~\ref{def:tw}. Moreover, by the strong maximum principle, it is enough to prove the large inequality $U \leq 1$. 

We proceed by contradiction and assume first that there exists $x_0 \in \mathbb{R}$ such that $U(x_0) = \sup_{\mathbb{R}} U > 1$. Evaluating the equation at $x_0$, we get 
$$0 \leq  - U '' (x_0) =  f(U(x_0-c\tau)) - U(x_0).$$
Thus $f(U (x_0 - c \tau)) >1 = \max_{[0,1]} f$, which enforces $U (x_0 - c \tau)>1$. Then $ U(x_0 - c \tau) >  f(U (x_0 - c \tau)) \geq U (x_0)$, a contradiction with our choice of $x_0$.

Next, consider the case when
$$L:=\limsup_{x\to+\infty} U(x) > 1.$$
Take a sequence $x_n \to +\infty$ such that $U(x_n) \to L$. By standard interior elliptic estimates and up to a subsequence, $U(\cdot+x_n)$ converges  locally uniformly to a $U_\infty$ which also solves
\begin{equation}\label{qsd_11}
-U_\infty '' + c U_\infty ' = f(U_\infty ( \cdot - c \tau)) - U_\infty \quad  \text{ in } \R,
\end{equation}
and satisfies $U_\infty  \leq L=U_\infty (0)$. Evaluating the equation at $0$ and proceeding as above, we again reach a contradiction. Therefore $L\leq 1$ and from the above paragraph $U\leq 1$, and thus $U<1$.

%
%

Let us now turn to the proof that $U(+\infty) = 1$. Proceed by contradiction and assume that
$$l:=\liminf_{x\to+\infty} U(x) < 1.$$
Take a sequence $x_n \to +\infty$ such that $U(x_n) \to l$. By standard interior elliptic estimates and up to a subsequence, $U(\cdot+x_n)$ again converges  locally uniformly to a solution~$U_\infty$ of~\eqref{qsd_11}, and such that $U_\infty  \geq l=U_\infty (0)$. Evaluating the equation at $0$, we get
\begin{equation}\label{to-be-denied}
0 \geq -U_\infty '' (0)  = f (U_\infty (- c \tau)) - U_\infty (0).
\end{equation}
By the definition of a travelling wave --- in particular $\liminf _{x\to +\infty} u(x)>\theta$--- and because we have just proved that $U<1$, there exists $\delta>0$ small enough such that $\theta + \delta < U_\infty \leq 1$. Moreover, by construction $U_\infty (0) < 1$. 
Then either $U_\infty (- c \tau)) =1$, in which case $f(U_\infty (-c\tau)) = 1$, or $U_\infty ( -c \tau) \in (\theta,1)$ in which case $ f(U_\infty (-c \tau)) > U_\infty (-c\tau)$. In both cases it follows that $f(U_\infty(-c \tau)) - U_\infty (0) >0$, a contradiction with \eqref{to-be-denied}. The proposition is proved.
\end{proof}

\subsection{When the delay is small}\label{ss:small-delay}

We consider here the case of small delay and  prove Theorem \ref{th:small-delay} under assumption \eqref{small-delay}. Using $L^2$ estimates (see \cite{Alf-Cov} for related arguments), we find a sufficient
condition for a bistable travelling wave to converge to 1 at $+\infty$.
Let us start with the following lemma.

\begin{lemma}[Sufficient condition for $u'\in L^2$]\label{lem:uL2}
Let $(c,U)\in\R\times C^2_b(\R)$ be a bistable travelling wave.  Assume $
\tau\Vert f'\Vert _{L^\infty(0,\Vert U \Vert _\infty)}<1$.

Then $U' \in
L^2(\R)$ and $U'(\pm\infty)=0$.\end{lemma}

\begin{proof}  Let us denote $M_0:=\Vert U\Vert _\infty$, and $M_1:=\Vert U'\Vert _\infty$. We rewrite the equation as
$$
cU'(x)=U''(x)+f(U(x))-U(x)+f(U(x-c\tau))-f(U(x))\,, $$
multiply it by $U'$, and then
integrate from $-A<0$ to $B>0$ to get
\begin{equation}\label{eq1}
c\int _{-A}^B {U'}^2=\left[\frac 12 {U'}^2+F(U)-\frac 12 U^2\right]_{-A}^B+\int_{-A}^B(f(U(x-c\tau))-f(U(x)))U'(x)dx,
\end{equation}
where $F$ denotes a primitive of $f$. We denote by $I_{A,B}$ the last integral appearing above and use
the Cauchy-Schwarz inequality to see
\begin{equation}\label{eq2}
{I_{A,B}}^2\leq \int _{-A}^B U'^2(x)dx \int _{-A}^B (f(U(x-c\tau))-f(U(x)))^{2}dx.
\end{equation}
Now, for a given $x$, we write
$$
f(U(x-c\tau))-f(U(x))=c\tau \int _0  ^{1} -f'(U(x-c\tau z))U'(x-c\tau z)dz,
$$
so that another application of the Cauchy-Schwarz inequality
yields
\begin{eqnarray*}
(f(U(x-c\tau))-f(U(x)))^2& \leq &c ^{2}\tau ^{2} \int_0 ^{1}f'^{2}(U(x-c\tau z))dz\int _0^{1}U'^{2}(x-c\tau z)dz\\
&\leq& c ^{2}\tau ^{2} \Vert f'\Vert _{L^{\infty}(0,M_0)}^{2}\int _0^{1}U'^{2}(x-c\tau z)dz.
\end{eqnarray*}
Integrating this we find
$$
\int_{-A}^B (f(U(x-c\tau))-f(U(x)))^{2}dx  \leq c ^{2}\tau ^{2} \Vert f'\Vert _{L^{\infty}(0,M_0)}^{2}\int _0^{1}\int _{-A-c\tau z}^{B-c\tau z}U'^{2}(y)dydz.
$$
By cutting into three pieces, we get
$$
\int _{-A-c\tau z}^{B-c\tau z}U'^{2}(y)dy
\leq   \int_{-A}^{B}{U'}^2+2\vert c\vert \tau M_1^2 z,$$
which in turn implies
\begin{equation}\label{eq4}
\int_{-A}^B (f(U(x-c\tau))-f(U(x)))^{2}dx  \leq c ^{2}\tau ^{2} \Vert f'\Vert _{L^{\infty}(0,M_0)}^2 \left(\int _{-A} ^{B}U'^{2}+\vert c\vert \tau M_1^{2}\right).
\end{equation}
If $R_{A,B}:=\int_{-A}^{B}{U'}^2$, combining (\ref{eq1}),
(\ref{eq2}) and (\ref{eq4}) we see that
$$
|c| R_{A,B} \leq \left(M_1^{2}+2\Vert F\Vert _{L^{\infty}(0,M_0)}+M_0^{2}\right)+\vert c\vert \tau \Vert f'\Vert _{L^{\infty}(0,M_0)}\sqrt{R_{A,B}(R_{A,B}+\vert c\vert \tau M_1^{2})}
$$
Since $\tau\Vert f'\Vert _{L^\infty(0,M_0)}<1$ and $c\neq 0$ (see the beginning of  Section \ref{s:properties}), the upper estimate compels
$R_{A,B}=\int_{-A}^B {U'}^2$ to remain bounded, so that $U'\in
L^2$. Since $U'$ is uniformly continuous on $\R$, this implies
$U'(\pm\infty)=0$. This concludes the proof of the lemma.
 \end{proof}

 We are now in the position to complete the proof of Theorem \ref{th:small-delay}.

\begin{proof}[Proof of Theorem \ref{th:small-delay} under assumption \eqref{small-delay}]
 Under the assumptions of Lemma~\ref{lem:uL2} above, denote by $\mathcal A$ the set of
accumulation points of $U$ in $+\infty$. Since $\liminf _{x\to
\infty}U(x)>\theta$, we have $\mathcal A \subset(\theta,\Vert U
\Vert _\infty]$. Let $l \in \mathcal A$. There is $x_n \to
\infty$ such that $U(x_n)\to l$. Then $v_n(x):=U(x+x_n)$ solves
$$
-{v_n}''+c{v_n}'=f(v_n(\cdot -c\tau))-v_n \quad \textrm{ on }
\R\,.
$$
Since the $L^\infty$ norm of the right hand side member is
uniformly bounded with respect to $n$, the interior elliptic
estimates imply that, for all $R>0$, all $1<p<\infty$, the
sequence $(v_n)$ is bounded in $W^{2,p}([-R,R])$. From Sobolev
embedding theorem, one can extract $v_{\varphi (n)} \to v$
strongly in $C^{1,\beta}_{loc}(\R)$ and weakly in $W^{2,p}_{loc}
(\R)$. It follows from Lemma \ref{lem:uL2} that
$$
v'(x)=\lim _{n\to\infty} U'(x+x_{\varphi(n)})=0,
$$
so that $v\equiv 0$ or $v\equiv \theta$ or $v\equiv 1$. From
$v(0)=\lim _n U(x_{\varphi(n)})=l>\theta$ we deduce that $l=1$.
Therefore $U(+\infty)$ exists and is equal to $1$.
\end{proof}

\section{Oscillations around 1}\label{s:oscill}

The last section deals with a large delay case where oscillations around~1 occur, that is we prove Theorem \ref{th:large-delay}.  We refer the reader to \cite[Theorem 3]{Trofimchuk} for a similar oscillations result in a KPP situation, and from which  we borrow some of the arguments. We start with some preparations.

\subsection{Preliminary on related monotone problems}\label{ss:prelim}

This subsection is devoted to the presentation of auxiliary results, related to bistable nonlocal problems with monotony that will be used in the proof of Theorem \ref{th:large-delay}. To that aim we consider
\begin{equation}\label{g-monotone}
\text{  a smooth nondecreasing function $g:[0,1]\to [0,1]$},
\end{equation}
as well as the function $G$ defined by $G(u):=g(u)-u$.
We assume that $G$ is bistable between $0$ and $1$ in the sense that there exists $\theta\in (0,1)$ such that
\begin{equation}\label{hyp-bistable}
\begin{cases}
G(0)=G(\theta)=G(1)=0,\\
G'(0)<0,\;G'(1)<0\text{ and }G'(\theta)>0,\\
G(u)<0\text{ for $u\in (0,\theta)$ and $G(u)>0$ for $u\in (\theta,1)$}.
\end{cases}
\end{equation}
Together with this assumption, we consider the following nonlocal problem
\begin{equation}\label{eq-shift}
\left(\partial_t-\partial_{xx}\right)v(t,x)=g\left(v(t,x-h)\right)-v(t,x),\;\;t>0,\;x\in\R,
\end{equation}
for some shift parameter $h\in \R$.

Our preliminary result reads as follows.

\begin{lemma}[On the monotone problem with shift] \label{lem:tw-monotone2}
Let \eqref{g-monotone} and \eqref{hyp-bistable} hold. Then the following properties hold.
\begin{itemize}
\item[$(i)$] Let $h\in\R$ be given. Then problem \eqref{eq-shift} admits a unique (up to translation) travelling wave solution $\left(c_h,U_h\right)\in \R\times C^2_b(\R)$ such that
\begin{equation*}
\lim_{x\to-\infty} U_h(x)=0,\;\lim_{x\to+\infty} U_h(x)=1,\text{ and }U_h'>0.
\end{equation*}
\item[$(ii)$] Let  $h\in\R$ be given.  If $v_0\in L^\infty(\R)$ satisfies $0\leq v_0\leq 1$ and
\begin{equation*}
\limsup_{x\to-\infty}v_0(x)<\theta\text{ and }\liminf_{x\to+\infty} v_0(x)>\theta,
\end{equation*}
then the solution $v=v(t,x)$ of equation \eqref{eq-shift} supplemented with the initial datum $v_0$ satisfies 
\begin{equation*}
\sup_{x\in\R}\left|v(t,x)-U_h\left(x+c_ht+\xi\right)\right|=O\left(e^{-\kappa t}\right)\text{ as }t\to\infty, 
\end{equation*}
wherein $\xi$ is some constant depending on the initial datum $v_0$, and $\kappa>0$ is some given constant (independent of $v_0$).
\item[$(iii)$] The wave speed satisfies $c_h\to c_0$ as $h\to 0$, wherein $c_0\in\R$ is the unique wave speed associated with the bistable (local) travelling wave problem
\begin{equation*}
\begin{cases}
-U''(x)+c_0U'(x)=G\left(U(x)\right),\;x\in\R,\\
U(-\infty)=0,\;U(+\infty)=1,\text{ and }U'>0.
\end{cases}
\end{equation*}
\end{itemize}
\end{lemma}

\begin{rem}\label{REM1}It is well-known that the speed $c_0$ defined above is positive when $\int_0^1 G(u)du>0$. Hence, because of  $(iii)$, in that case one gets
$$
c_h>0 \text{ for all $|h|<<1$}.
$$
\end{rem}

\begin{proof} Items $(i)$ and $(ii)$ follow from the results of Chen \cite{Chen}. It thus remains to prove $(iii)$.
To that aim let us first notice that the family $\{c_h\}_{h\in\R}$ is decreasing with respect to $h$. This directly follows from the stability result stated in $(ii)$ combined with the comparison principle (recall that $g$ is nondecreasing on $[0,1]$) and the monotony of travelling waves. Next we take a sequence $h_n\to 0$ and need to show $c_{h_n}\to c_0$ as $n\to\infty$. To do so,  let us choose two normalisation sequences $\left\{\xi_n^1\right\}$ and $\left\{\xi_n^2\right\}$ such that, for all $n\geq 0$, 
\begin{equation*}
U_{h_n}(\xi_n^1)=\frac{\theta}{2} \ \text{ and }  \ U_{h_n}(\xi_n^2)=\frac{1+\theta}{2}.
\end{equation*}
Then we define the functions
\begin{equation*}
u_n(x):=U_{h_n}(x+\xi_n^1)\  \text{ and } \ v_n(x):=U_{h_n}(x+\xi_n^2),
\end{equation*}
so that
\begin{equation*}
u_n(0)=\frac{\theta}{2}\  \text{ and } \ v_n(0)=\frac{1+\theta}{2}.
\end{equation*}
Next, due to elliptic regularity, and possibly along a subsequence, we can assume  $\left(u_n,v_n\right)(x)\to (u,v)(x)$ locally uniformly up to their second order derivatives, and also  that $c_{h_n}\to c^*\in\R$.
As a consequence, the functions $u$ and $v$ are nondecreasing and satisfy
\begin{equation*}
\begin{split}
&-u''(x)+c^* u'(x)=G\left(u(x)\right),\;x\in\R,\\
&-v''(x)+c^* v'(x)=G\left(v(x)\right),\;x\in\R,\\
\end{split}
\end{equation*}
together with the normalisation conditions
\begin{equation*}
u(0)=\frac{\theta}{2} \ \text{ and } \ v(0)=\frac{1+\theta}{2}.
\end{equation*}
The monotonicity properties for $u$ and $v$ ensure that
\begin{equation*}
\begin{split}
&u(-\infty)=0\text{ and }u(+\infty)\in\{\theta,1\},\\
&v(-\infty)\in\{0,\theta\}\text{ and }v(+\infty)=1.
\end{split}
\end{equation*}
If $u(+\infty)=1$ or $v(-\infty)=0$, then $c^*=c_0$ and the result follows.
It is therefore sufficient to exclude the situation where $u(+\infty)=\theta$ and $v(-\infty)=\theta$. 
However, multiplying the $u-$ and $v-$equations by $u'$ and $v'$ respectively and integrating over $\R$ yield
 $c^*<0$ from the monostable $u$-equation and $c^*>0$ from the monostable $v$-equation, a contradiction. The lemma is proved.
\end{proof}

\subsection{Proof of Theorem \ref{th:large-delay} on oscillations}\label{ss:osci}

In this subsection, let Assumption \ref{ass:f-osc} hold. Setting
\begin{equation*}
M_1:=f(\beta)\text{ and }M_2=\left(f\circ f\right)(\beta),
\end{equation*}
Assumption \ref{ass:f-osc} implies in particular that $M_1>1$, $\theta<\alpha<M_2<1$, and
\begin{equation*}
f(u)\begin{cases}  >1\ \ \text{ for }u\in [M_2,1),\\<1 \ \ \text{ for }u\in (1,M_1].\end{cases}
\end{equation*}

In the sequel, we consider $(c,U) \in \R\times C^2 _b(\R)$ a bistable travelling wave of \eqref{eq} provided by Theorem \ref{th:construction}. This means, in particular, that $0<U(x)<M_1$ for all $x\in \R$, and
\begin{equation}\label{propri-wave}
\begin{cases}
c>0,\; U(0)=\theta,\\
U'(x)>0,\;\forall x\in (-\infty,c\tau],\\
U(x)>\theta,\;\forall x>0\text{ and }\displaystyle \liminf_{x\to+\infty}U(x)>\theta.
\end{cases}
\end{equation}
For notational simplicity we shall write $h:=c\tau>0$. 

Before going to the proof of Theorem \ref{th:large-delay}, we need  some basic qualitative properties of the travelling wave solution $(c,U)$. Our first lemma is concerned with the maximal interval of monotonicity of $U$. From \eqref{strict-croissance}, we can define
\begin{equation}\label{def-sigma*}
\sigma_*=\sup\left\{\sigma\in \R:\;U'(x)>0,\;\forall x\in (-\infty,\sigma)\right\} \in [h,+\infty].
\end{equation}
The proof of the following lemma could be borrowed from \cite[Theorem 11 and Theorem 13]{Trofimchuk}, but we propose an alternative proof that makes use of a sliding argument.

\begin{lemma}[Behaviour of $U$ at $\sigma _*$]\label{LE-max-monotone}
The following alternative holds.
\begin{itemize}
\item[$(i)$] If $\sigma_*=+\infty$ then $U(+\infty)=1$. 
\item[$(ii)$] If $\sigma_*<+\infty$ then  $U\left(\sigma_*\right)>1$. 
Moreover in that case, there exists $\sigma_{**}\in \left(\sigma_*,\infty\right]$ such that $U'(x)<0$ for all $x\in\left(\sigma_*,\sigma_{**}\right)$ and $U'(\sigma_{**})=0$.  If $\sigma_{**}<\infty$ then
\begin{equation*}
U\left(\sigma_{**}\right)<1<U\left(\sigma_{**}-h\right).
\end{equation*}
 In addition, if $\sigma_{**}<\sigma_*+h$ then $U'(x)>0$ on $\left(\sigma_{**},\sigma_{*}+h\right]$.
\end{itemize}
\end{lemma}

\begin{proof}
If $\sigma_* = +\infty$, then $U(+\infty)$ exists and it has to be equal to~1. Hence $(i)$ follows. 

We thus focus on the second statement and  assume $\sigma_*<+\infty$. Let us  recall that, from our construction, the wave profile $\left(c,U\right)$ considered in this subsection is given by
\begin{equation}\label{const}
c=\lim_{n\to\infty}c_{a_n}\; \text{ and } U(x)=\lim_{n\to\infty} u_{a_n}\left(x+x_{a_n}\right)\text{ locally uniformly in $\R$},
\end{equation}
wherein $a_n>0$ is a given sequence tending to $\infty$ as $n\to\infty$, $\left(c_{a_n},u_{a_n}\right)$ is a solution of problem $P_1(a_n)$ in a box with normalization \eqref{normalization} while $x_{a_n}>0$ is the leftmost point where $u_{a_n}\left(x_{a_n}\right)=\theta$ (see \eqref{def:xa}).
For each $n\geq 0$, let $x_\beta^n\in \left(x_{a_n},a_n\right)$ denote the leftmost point where $u_{a_n}(x_\beta^n)=\beta$.

\begin{claim}\label{claim8}
For all $n\geq 0$, we have
\begin{equation*}
u_{a_n}'(x)>0,\;\forall x\in (-a_n,x_\beta^n).
\end{equation*}
\end{claim}
\begin{proof}
We fix $n\geq 0$ and, for notational simplicity, we write $(c_n,u_n)$ and $x_n$ instead of, respectively, $\left(c_{a_n},u_{a_n}\right)$ and $x_{a_n}$. We use a sliding method: let us consider the real number $\xi^*\in [0,a_n+x_\beta^n)$ defined by
\begin{equation*}
\xi^*=\inf\left\{\xi\in [0,a_n+x_\beta^n]:\;u_n(x-\xi)<u_n(x),\;\forall x\in \left[-a_n+\zeta,x_\beta^n\right],\;\forall \zeta\in [\xi,a_n+x_\beta^n]\right\}.
\end{equation*}
Assume by contradiction $\xi^*\neq 0$.  Then there exists $x^*\in (-a_n+\xi^*,x_\beta^n]$ such that
\begin{equation*}
\begin{cases}
u_n(x-\xi^*)\leq u_n(x),\;\forall x\in \left[-a_n+\xi^*,x_\beta^n\right],\\
u_n(x^*-\xi^*)= u_n(x^*).
\end{cases}
\end{equation*}
Notice that $x^*$ cannot be equal to $x_\beta^n$ because $u_n (x_\beta^n - \xi^*) < \beta = u_n (x_\beta^n)$.

Consider now the function $w(x):=u_n(x)-u_n(x-\xi^*)$ and observe that it satisfies $w(x)\geq 0$ on $[-a_n+\xi^*,x_\beta^n]$, $w(x^*)=0$ as well as 
\begin{equation*}
-w''(x)+cw'(x)+w(x)=f\left(u_n(x-h)\right)-f\left(u_n(x-\xi^*-h)\right)\; \text{ for $x\in (-a_n+\xi^*,x_\beta^n)$}.
\end{equation*} 
Now observe that one has
$$ 
u_n(x-h)\geq u_n(x-\xi^*-h),\; \forall x\in \left[-a_n+\xi^*,x_\beta^n\right].
$$
Indeed, from the above, one already knows that 
$$
u_n(x-\xi^*-h)\leq u_n(x-h),\; \forall x\in \left[-a_n+\xi^*+h,x_\beta^n+h\right],
$$
whereas, if $x\in [-a_n+\xi^*,-a_n+\xi^*+h]$ then $u_n(x-\xi^*-h)=0$.
Finally, since $u_n(x-h)\leq \beta$ for $x\in [-a_n,x_\beta^n]$ and $f'>0$ on $(0,\beta)$, one obtains 
\begin{equation*}
\begin{cases}
-w''(x)+cw'(x)+w(x)\geq 0\; \text{ for $x\in (-a_n+\xi^*,x^n_{\beta})$},\\
w\geq 0\text{ and }w(x^*)=0.
\end{cases} 
\end{equation*}
Hence the strong comparison principle applies and ensures that
$$
u_n(x)=u_n(x-\xi^*),\;\forall x\in [-a_n+\xi^*,x_n^beta].
$$
Finally choosing $x=-a_n+\xi^*$ yields $0<u_n(-a_n+\xi^*)=u_n(-a_n)=0$, a contradiction. Hence $\xi ^*=0$, meaning that $u_n$ is increasing on $[-a_n,x_n^\beta]$. Hence $v_n:=u_n'\geq 0$ on $(-a_n,x_n^\beta)$, and $-v_n''+cv_n'+v_n=v_n(\cdot -h)f'(u_n(\cdot - h))\geq 0$ on $(-a_n,x_n^\beta)$, so that the strong maximum principle implies $u_n '>0$ on $(-a_n,x_n^\beta)$, thus completing the proof of Claim \ref{claim8}.
\end{proof}

\begin{claim}\label{claim7}
The function $U$ satisfies that there exists $x_\beta>0$ such that
\begin{equation*}
U(x_\beta)=\beta\text{ and }U(x)<\beta,\;\forall x\in (-\infty,x_\beta),
\end{equation*}
and $U'(x)>0$ for all $x<x_\beta $.
\end{claim}

\begin{proof}
{}From the previous step, let us first show that the sequence $\left\{x_\beta^n\right\}$ is bounded.
To that aim, we argue by contradiction by assuming that, up to a subsequence, $x_\beta^n\to\infty$. In that case, recalling \eqref{const}, one obtains that $U'(x)\geq 0$ and $U(x)\leq \beta<1$, for all $x\in\R$. This means that $U(x)\to\theta$ as $x\to\infty$, a contradiction with \eqref{propri-wave}.
Hence, $\left\{x_\beta^n\right\}$ is bounded and there exists $x_\beta>0$ such that $U(x_\beta)=\beta$, $U(x)\leq \beta$ for all $x\leq x_\beta$ and $U'(x)\geq 0$ for all $x\leq x_\beta$. Finally, the strong maximum principle applied to $U'$, as in the end of the above proof of Claim \ref{claim8}, ensures that $U'(x)>0$ for all $x\in (-\infty,x_\beta)$ and this completes the proof of Claim \ref{claim7}.
\end{proof}

The above claim ensures that $\sigma_*\geq x_\beta$.
Equipped with this key property, we are now able to complete the proof of Lemma \ref{LE-max-monotone}. First, assume by contradiction that $U (\sigma_*)\leq 1$. As in the end of the proof of Claim \ref{claim8}, observe that $U'\geq 0$ satisfies $-U'''+cU''+U'\geq 0$ on $(-\infty,\sigma _*)$ and $U'(\sigma*)=0$ so that the Hopf lemma implies $U''(\sigma _*)<0$.
Thus $\sigma_*$ is  a point of local maximum and there exists $\sigma_{**} \in ( \sigma_*,+\infty)$ such that $U' < 0$ in $(\sigma_*, \sigma_{**})$ and $U'(\sigma _{**})=0$, so that $U''(\sigma_{**})\geq 0$. Evaluating the equation at the point $\sigma_{**}$, we get 
\begin{equation}\label{eval**}
f (U (\sigma_{**} -h)) - U (\sigma_{**}) \leq 0.
\end{equation}
Let us now check that 
\begin{equation}\label{gap11}
\sigma_{**} -h  > x_\beta .
\end{equation}
Assume this is false. For all $x \in [\sigma_* , \sigma_{**}]$, 
$$-U''' (x)+ cU'' (x) = U' (x-h) f' (U(x-h))  - U'(x).$$
Since $x - h \leq x_\beta$, we have $U'(x-h) \geq 0$ and $f'(U(x-h)) \geq 0$, and by the definition of $\sigma_{**}$ we also have $U'(x) \leq 0$. Thus
$$-U''' + cU'' \geq 0$$
in the interval $[\sigma_*, \sigma_{**}]$. Since $U'(\sigma_*)= U'(\sigma_{**}) =0$, from the maximum principle we conclude that $U' \equiv 0$ in $[\sigma_*,\sigma_{**}]$, a contradiction. Now \eqref{gap11} is proved.

It follows that $\theta < U(\sigma_{**}- h ) \leq f(U (\sigma_{**} - h))$ and, from \eqref{eval**}, we infer $U (\sigma_{**}) \geq U(\sigma_{**} - h)$. This implies that $\sigma_{**}-h <  \sigma_*$. Thus $1 > U (\sigma_{**}) \geq U(\sigma_{**}-h) > \beta$ and, because $f$ decreases on the interval $(\beta,1)$, we get
$$f (U (\sigma_{**} -h)) \geq f (U (\sigma_{**}))>U(\sigma_{**}),$$
which contradicts the inequality~\eqref{eval**}. Thus $U(\sigma_*) >1$.

Let us now assume $\sigma _{**}<\infty$. Similarly as above, we define $\sigma_{***} \in (\sigma_{**},+\infty]$ such that $U' >0$ on $(\sigma_{**},\sigma_{***})$ and $U' (\sigma_{***}) = 0$ if $\sigma_{***}$ is finite. Let us check that
\begin{equation}\label{gap12}
\sigma_{***} -h  > \sigma_* .
\end{equation}
We only need to consider the case when $\sigma_{***} < \infty$ and $\sigma_{**} -h < \sigma_*$, and then we proceed similarly as above. For any $x\in [\sigma_{**},\overline{\sigma} := \min \{ \sigma_{***}, \sigma_* + h\}]$ we have $U'(x-h)\geq 0$, $U'(x)\geq 0$ and $f'(U(x-h))\leq 0$ since $ U(x-h)\geq \beta$. As a result,  $f(U(x - h))-U(x)$ is nonincreasing with respect to $x$ between $\sigma_{**}$ and $\overline \sigma$. As it is nonpositive at the point $\sigma_{**}$, it remains nonpositive on $[\sigma_{**},\overline{\sigma}]$. It follows that $U$ is convex on $[\sigma_{**},\overline \sigma]$, and thus $U' (\overline{\sigma}) >0$. Since $U'(\sigma_{***})=0$, this means that  $\overline{\sigma}=\sigma_*+h$ and \eqref{gap12} is proved. Note that the last part of Lemma~\ref{LE-max-monotone} $(ii)$ immediately follows.

Next we  prove that $U(\sigma_{**}) < 1$.  If not, then it is straightforward that $\sigma_{***}$ must be finite and $U(\sigma_{***}) >1$. In particular, it is a local maximum of $U$. Evaluating the equation at $\sigma_{***}$, we get 
$$f(U(\sigma_{***} -h)) \geq U(\sigma_{***}) \geq 1.$$
Moreover, here and from~\eqref{gap12} one must have $U(\sigma_{***}-h) >1$, and the previous inequality now implies that $U(\sigma_{***}-h) > U(\sigma_{***})$, both values being in the interval $[1, M_1= f(\beta)]$. Using the monotonicity of $f$ in this same interval, one gets
$$f(U(\sigma_{***}-h)) < f(U(\sigma_{***})) <1,$$
a contradiction.

Last, we check that $U(\sigma_{**} -h) >1$. Evaluating the equation at $\sigma_{**}$, we get that 
$$f(U(\sigma_{**} -h)) \leq U(\sigma_{**}) <1.$$
Hence either $U(\sigma_{**} -h) \in [0,\alpha]$, which contradicts~\eqref{gap11} and the fact that $\alpha < \beta$, or $U(\sigma_{**}-h) >1 $. This completes the proof of Lemma~\ref{LE-max-monotone}.
\end{proof}

{}From the above lemma, we can now trap the solution $U$ between $M_2$ and $M_1$, after $\sigma _*$.

\begin{lemma}[Bounds on $U$ after $\sigma _*$]\label{LE-localisation}
If $\sigma_*<+\infty$ then 
$
M_2\leq U(x)\leq M_1$ for all $ x\geq \sigma_*$.
\end{lemma}

\begin{proof}
Note that the upper bound has already been stated above. For the lower bound, proceed  by contradiction and assume that $U(x_0)< M_2$ for some $x_0 > \sigma_*$. Without loss of generality, $x_0$ may be the leftmost point in $(\sigma_*,+\infty)$ where $U$ reaches the value $U(x_0)$, so that $U' (x_0) \leq 0$. Then either $U' \leq 0$ on $[x_0 , +\infty)$, or $U$ has a local minimum with value strictly smaller than $M_2$ in $[x_0,+\infty)$. In the former case, the limit $U(+\infty)$ exists and satisfies $f(U(+\infty))= U(+\infty)$, together with $\theta <U(+\infty)<M_2<1$, which is impossible. Then consider the latter case, and denote by $t_0$ the leftmost local minimum point in $(\sigma_*,+\infty)$ (and therefore in $\mathbb{R}$) such that $U(t_0)<M_2$. Using the notation of the previous lemma, this means that $\sigma_{**}<\infty$ and $t_0\geq \sigma_{**}$ while
$$
\theta<U(t_0)<M_2,\;U'(t_0)=0\text{ and }U''(t_0)\geq 0.
$$
Plugging this information into the wave equation yields $f\left(U(t_0 -h)\right)\leq U(t_0)$.  Together with Assumption~\ref{ass:f-osc} this ensures that $U(t_0-h)\leq M_2$. Also, from the second statement of Lemma \ref{LE-max-monotone}, one must have $t_0 > \sigma_{**}$ and $t_0 -h > \sigma_*$, so that $U(t_0-h)>\theta$. From Assumption \ref{ass:f-osc}, we deduce
\begin{equation*}
U(t_0-h)<f\left(U(t_0 -h)\right)\leq U(t_0).
\end{equation*}
Since $t_0 -h > \sigma_*$, the above inequality contradicts the definition of $t_0$. This completes the proof of Lemma \ref{LE-localisation}. 
\end{proof}

{}From the above lemmas we now provide a sufficient condition, which involves the wave speed~$c$ and the delay parameter $\tau$, ensuring that the wave profile $U$ does not converge to $1$ at $x \to + \infty$. This result is related to the location of the complex roots of the function $\Delta  ^{\mu} (\cdot;h;\tau)$ defined by
\begin{equation}\label{def-Delta}
\Delta^\mu(\lambda;h;\tau):=\frac{\lambda^2}{h^2}-\frac{\lambda}{\tau}+\mu e^{-\lambda }-1, \; \lambda \in \C,
\end{equation}
where $\mu \in \R$, $h>0$, $\tau >0$ are parameters. Notice that \eqref{def-Delta} arises by plugging the ansatz $1+\ep e^{\frac \lambda h x}$ in the travelling wave equation and keeping the $\ep$ order terms, with $\mu=f'(1)$.

\begin{rem}\label{rem-prel} 
Notice that letting $\lambda=\tau z$, equation \eqref{def-Delta} is transferred into the equation of \cite[Lemma 17]{Trofimchuk}. As easily seen, for any parameters $\mu<0$, $h>0$, $\tau>0$, a solution $\lambda$ of $\Delta ^\mu (\lambda;h;\tau)=0$ which is not simple has to be real and is at most double. Also, from the computation in the proof of item (3) in \cite[Lemma 17]{Trofimchuk}, we have that if $\xi _1+i\xi_2$ is a solution then on the vertical line $\Re (\lambda)=\xi_1$ the only other possible solution is $\xi _1-i \xi _2$. As a result, for any parameters  $\mu<0$, $h>0$, $\tau>0$, we have that, for each $\xi_1 \in\R$, the equation $\Delta^\mu(\lambda;h;\tau)=0$ has at most two solutions (counting multiplicity) on the vertical line $\Re (\lambda)=\xi_1$.
\end{rem}

\begin{lemma}[Sufficient condition for not converging to 1]\label{LE-behaviour}
Setting $\mu=f'(1)<0$ and $h=c\tau$, if the equation $\Delta^\mu (\lambda;h;\tau)=0$ does not admit any solution in the complex strip $S_0$ then the wave profile $U$ does not converge to $1$ as $x \to + \infty$.
Here the strip $S_0$ is given by
\begin{equation}\label{strip}
S_0=\left\{\lambda \in\mathbb C : \Re(\lambda)\leq 0\text{ and }\Im(\lambda)\in\left[-2\pi,2\pi\right]\;\right\}.
\end{equation}
\end{lemma} 

\begin{proof}[Sketch of proof]
The proof of this result is similar to the one of  \cite[Theorem 3]{Trofimchuk} which deals with a KPP situation. Here we provide a sketch of proof using slightly different arguments.

First, under the hypothesis of the  lemma, one can reproduce the proof of \cite[Lemma 25]{Trofimchuk} and derive that $U$ cannot be eventually monotone, so that in particular $\sigma_*<\infty$.

Next, consider the functions $z_0(t)=U(t)-1$ and $z_1(t)=U'(t)$ and note that they satisfy the delayed cyclic system of equations
\begin{equation}\label{system-z}
\begin{cases}
z_0 ' (t)=z_1(t),\\
z_1'(t)=c z_1(t)+z_0(t)+1-f\left(z_0(t-h)+1\right) =: g(z_0 (t),z_1(t),z_0(t-h)).
\end{cases}
\end{equation} 
Let us observe that the above system exhibits a delayed cyclic structure with a positive feedback (see Mallet-Paret and Sell \cite{Malet-Paret}). In our situation, positive feedback means that the function $t\mapsto g(z_0(t),0,z_0(t-h))$ is nonnegative (resp. nonpositive) whenever $z_0 (t)$ and $z_0 (t-h)$ are both nonnegative (resp. nonpositive). This is true for all $t \geq \sigma_*+h$ due to Lemma~\ref{LE-localisation} which insures that $U(x)=z_0(x)+1$ is trapped between $M_2>\alpha$ and $M_1$ for all $x\geq \sigma_*$. 

In this positive feedback case, Mallet-Paret and Sell \cite{Malet-Paret} developed  a discrete Lyapunov functional that allows to control the number of sign changes of the solutions. 
We now recall the definition of the number of sign changes as proposed  in \cite{Malet-Paret}: for any function $\varphi\in C^0(\mathbb K)$, $\varphi \not \equiv 0$, with $\mathbb K:=[-h,0]\cup \{1\}$,  we define its number of sign changes, denoted by ${\rm sc}\,(\varphi)$, by
\begin{equation*}
{\rm sc}\,(\varphi):=\sup\left\{k\geq 0:\;\exists \left(t_j\right)_{j=0}^k\in\mathbb K:\;t_0<..<t_k,\;\;\varphi\left(t_j\right)\varphi\left(t_{j+1}\right)<0,\;\forall 0\leq j \leq k-1\right\},
\end{equation*}
and ${\rm sc}\,(\varphi)=0$ whenever $\varphi \geq 0$ or $\varphi\leq 0$. We now define the function $\Psi:t\in \R\mapsto \Psi_t\in C^0(\mathbb K)$ by
\begin{equation*}
\Psi_t(\theta)=\begin{cases} z_0(t+\theta)&\text{ if }\theta\in [-h,0],\\ z_1(t) &\text{ if }\theta=1.\end{cases}
\end{equation*}
If $\Psi _{t_0}\equiv 0$ for some $t_0$ then $U\equiv 1$ on $[t_0-h,t_0]$ and, using the equation, on $(-\infty,t_0]$, which is not. Hence, for any $t\in\R$, one has $\Psi_t\not\equiv 0$.  Lemma \ref{LE-max-monotone} then ensures, since $\sigma_*<\infty$, that ${\rm sc}\,\left(\Psi_{\sigma_*+h}\right)\leq 2$. Hence due to \cite{Malet-Paret}, this property is preserved in time, namely 
\begin{equation}\label{sign-changes}
{\rm sc}\,\left(\Psi_t\right)\leq 2,\;\text{ for all $t\geq \sigma_*+h$}.
\end{equation}

The rest of the proof is identical to the one of \cite[Theorem 3]{Trofimchuk} but, as mentioned above,  we propose a sketch of proof based on  different arguments, as those used in \cite{ducrot} to deal with a specific class of second order differential equations with distributed delay (namely a delayed gamma function like kernel). We argue by contradiction by assuming that $U(x)\to 1$ as $x\to+\infty$, or equivalently --- by elliptic estimates--- that $(z_0,z_1)(t)\to (0,0)$ as $t\to\infty$, where $(z_0,z_1)$ is the solution of \eqref{system-z}.

\begin{claim}\label{claim0}
The convergence $(z_0,z_1)(t)\to (0,0)$ as $t\to\infty$ is not super-exponential. Here recall that we say that this convergence is super-exponential if $e^{\kappa t}\left(z_0(t),z_1(t)\right)\to (0,0)$ as $t\to\infty$, for any $\kappa\in\R$.
\end{claim}

\begin{proof}  By adapting the proof of  \cite[Theorem 3.5]{ducrot} to the specific and simpler case of system \eqref{system-z}, one obtains  that if $(z_0,z_1)(t)$ has a super-exponential decay to $0$ as $t\to\infty$, then ${\rm sc}\,\left(\Psi_t\right)\to \infty$ as $t\to\infty$. This is in contradiction with \eqref{sign-changes}. 
\end{proof}

Hence, the function $t\mapsto \|\Psi_t\|$ converges to $0$ as $t\to\infty$ but this convergence is not super-exponential. From \cite[Corollary 24]{Trofimchuk} we infer that, if we fix $\rho>3h$, there exist some constant $D>1$ and a sequence $\{t_j\}_{j\geq 0}$ going to $\infty$ as $j\to \infty$ such that, for all $j\geq 0$,
\begin{equation*}
\|\Psi_{t_j}\|=\sup_{t\geq t_j}\|\Psi_t\|\, \text{ and }\sup_{t_j-\rho\leq t\leq t_j}\|\Psi_t\|\leq D \|\Psi_{t_j}\|.
\end{equation*}
Herein the symbol $\|.\|$ denotes the sup-norm in $C^0\left(\mathbb K\right)$.

Next we consider the sequence of functions $\left(y_{0,j},y_{1,j}\right)$, for $j\geq 0$, defined by
\begin{equation*}
y_{k,j}(t)=\frac{z_k(t+t_j)}{M_j},\;t\in\R,\,k=0,1\text{ with }M_j:=\|\Psi_{t_j}\|.
\end{equation*}	
Observe that the above properties imply
\begin{equation}\label{estim-ykj}
\vert y_{k,j}(t)\vert \leq D, \quad \forall k=0,1, \forall j\geq 0, \forall t \geq -3h.
\end{equation}
Moreover, for $j\geq 0$ and $t\in \R$, we define  
 $\Psi_t^j\in C^0(\mathbb K)$ by
\begin{equation*}
\Psi_t^j(\theta):=\begin{cases} y_{0,j}(t+\theta)&\text{ if }\theta\in [-h,0],\\ y_{1,j}(t) &\text{ if }\theta=1.\end{cases}
\end{equation*}
Hence, for all $j\geq 0$, there exists $\theta_j\in\mathbb K$ such that $|\Psi_0^j(\theta_j)|=1$, and
\begin{equation}\label{esti-osci-encore}
{\rm sc}\,\left(\Psi_t^j\right)\leq 2,\;\forall t\geq \sigma_*+h-t_j,\;\forall j\geq 0.
\end{equation}
Also, $\left(y_{0,j},y_{1,j}\right)$ satisfies the system
\begin{equation}\label{system-y}
\begin{cases}
y_{0,j} ' (t)=y_{1,j}(t),\\
y_{1,j}'(t)=c y_{1,j}(t)+y_{0,j}(t)+\widehat f\left(z_{0}(t+t_j-h)\right)y_{0,j}(t-h),
\end{cases}
\end{equation} 
wherein we have set 
\begin{equation*}
\widehat f(z)=\begin{cases}z^{-1}\left(1-f\left(z+1\right)\right)&\text{ if $z\neq 0$}\\ -f'(1) &\text{ if $z=0$}.\end{cases}
\end{equation*}
Hence, from \eqref{estim-ykj}, the family of functions $\{(y_{0,j},y_{1,j})\}_{j\geq 0}$ is uniformly bounded in $C^1([-2h,\infty))$ and, possibly along a sub-sequence, one may assume that
\begin{equation*}
(y_{0,j},y_{1,j})(t)\to \left(y_{0,*},y_{1,*}\right)(t)\text{ locally uniformly for $t\in [-2h,\infty)$ as $j\to\infty$},
\end{equation*}
and $(y_{0,*},y_{1,*})$ is a bounded solution of the problem
\begin{equation}\label{system-y*}
\begin{cases}
y_{0,*} ' (t)=y_{1,*}(t),\\
y_{1,*}'(t)=c y_{1,*}(t)+y_{0,*}(t)-f'\left(1\right)y_{0,*}(t-h),
\end{cases} t\geq -h.
\end{equation}
Since $|\Psi_0^j(\theta_j)|=1$, for all $j\geq 0$, one also knows that $\sup_{\theta\in [-h,0]}|y_{0,*}(\theta)|+|y_{1,*}(0)|\neq 0$. As a consequence of  \cite[Theorem 3.1]{Hale}, one derives that $\left(y_{0,*},y_{*,1}\right)$ is not a small solution, in the sense that this function does not super-exponentially converge to $0$ as $t\to\infty$.

As a consequence, due to the results in \cite{MP},  one derives that, for $\nu>0$  sufficiently large, it holds that
\begin{equation*}
\begin{pmatrix}
y_{0,*}\\y_{1,*}
\end{pmatrix}(t)=Y(t)+O\left(e^{-\nu t}\right),
\end{equation*}
where $Y(t)$ is a nonzero finite sum of eigenfunctions of the linear equation \eqref{system-y*} associated to eigenvalues in the strip $\left\{\lambda\in\mathbb C:\;\Re(\lambda)\in (-\nu,0]\right\}$.  Hence, using Remark \ref{rem-prel}, we deduce that 
\begin{equation*}
y_{0,*}(t)=Ae^{-\frac{\gamma}{h} t}\left[\cos\left(\frac{\omega}{h} t+\varphi \right)+o(1)\right],
\end{equation*}
where $A\in\R\setminus\{0\}$, $\varphi \in \R$, $\gamma\geq 0$, $\omega\in\R$ are constants such that $-\gamma+i\omega$ is a solution of ~\eqref{def-Delta}.
Finally, recalling that $y_{0,j}$ converges to $y_{0,*}$ as $j\to\infty$, the assumption of Lemma \ref{LE-behaviour} on the location of the eigenvalues, namely $|\omega|>2\pi$, ensures that $y_{0,j}$ crosses $0$ at least three times on each interval of the form $[t-h,t]$ when $t$ is large enough and $j$ large enough. This contradicts \eqref{esti-osci-encore} and completes the proof of Lemma \ref{LE-behaviour}.
\end{proof}

\begin{rem}\label{rem:wavetrain} At this point and under a slightly stronger hypothesis than Assumption \ref{ass:f-osc}, we can actually prove the convergence of the wave to a wavetrain as $x\to+\infty$. This is postponed to subsection \ref{ss:wavetrain}.
\end{rem}

The completion of the proof  of Theorem \ref{th:large-delay} now reduces to checking that the assumption of Lemma \ref{LE-behaviour}, about the location of the roots of the function $\Delta ^{\mu} (\cdot;h;\tau)$, is satisfied. We start with the following.

\begin{lemma}[On the roots of $\Delta ^{\mu} (\cdot;h;\tau)$]\label{LE-location}
Recall that we have set $\mu =f'(1)$ and that the strip $S_0$ was defined in \eqref{strip}. We assume that $\mu<-1$ and define
\begin{equation*}
\kappa_\mu :=\inf_{\lambda\in S_0 \setminus \{0 \} }\left|\frac{1-\mu e^{-\lambda}}{\lambda^2}\right|.
\end{equation*}
Then $\kappa_\mu>0$ and, for each $\varepsilon>0$, there exists $\tau_\varepsilon>0$ large enough such that the following holds true:
for all $(h,\tau)\in (0,\infty)\times(0,\infty)$ one has
\begin{equation*}
\begin{cases}
\tau>\tau_\varepsilon,\\
h\geq h_\mu+\varepsilon
\end{cases}
\Rightarrow
\left\{\lambda\in\mathbb C:\;\Delta^\mu \left(\lambda;h;\tau\right)=0\right\}\cap S_0=\emptyset.
\end{equation*}
Herein we have set $h_\mu=\left(\kappa_\mu\right)^{-\frac{1}{2}}$.
\end{lemma}

\begin{proof} The positivity of $\kappa_\mu$ easily follows from the fact that $\mu < -1$. Now to prove the  lemma, we fix $\varepsilon>0$ and let us argue by contradiction by assuming that there exist three sequences $\tau_n\to \infty$, $h_n\geq h_\mu+\varepsilon$ and $\lambda_n\in S_0$ such that
 $\Delta^\mu \left(\lambda_n;h_n;\tau_n\right)=0$. This re-writes as
\begin{equation}\label{pq}
e^{\lambda_n}\left[\frac{\lambda_n^2}{h_n^2}-\frac{\lambda_n}{\tau_n}\right]+\mu-e^{\lambda_n }=0,\;\forall n\geq 0.
\end{equation}
Let us first observe that, since $\mu\neq 0$, the sequence $\{\lambda_n\}$ is bounded. Indeed, first note that, since $\mu<-1$, $\lambda_n\neq 0$. Also note that the above equality re-writes as
\begin{equation*}
\mu\frac{e^{-\lambda_n }}{\lambda_n^2}=\frac{1}{\lambda_n^2}-\frac{1}{h_n^2}+\frac{1}{\lambda_n\tau_n},\;\forall n\geq 0,
\end{equation*}
that ensures that the sequence $\{\lambda_n\}$ is bounded.
Hence, up to a subsequence that is still denoted by $\{\lambda_n\}$, one may assume that $\lambda_n\to \lambda\in S_0$.
Now, note that the sequence $\{h_n\}$ is also bounded. Indeed, if it were unbounded then, up to a subsequence one may assume that $h_n\to\infty$ and, passing to the limit $n\to\infty$ into \eqref{pq} yields
\begin{equation*}
e^\lambda=\mu.
\end{equation*}
However, since $\mu<-1$ and $\lambda\in S_0$ the above equality cannot hold true. Hence the sequence $\{h_n\}$ is bounded and, up to a subsequence, one may assume that $h_n\to h\in \left[h_\mu+\varepsilon,\infty\right)$. Next passing to the limit $n\to\infty$ into \eqref{pq} yields
\begin{equation*}
e^{\lambda}\frac{\lambda^2}{h^2}+\mu -e^{\lambda}=0.
\end{equation*}
Let us also observe that, since $\mu < -1$ then $\lambda\neq 0$, so that
\begin{equation*}
\frac{1}{h^2}\in \Gamma\left(S_0\setminus\{0\}\right)\text{ with }\Gamma(\lambda)=\left|\frac{1-\mu e^{-\lambda}}{\lambda^2}\right|.
\end{equation*}
This means $h^{-2}\geq \kappa_\mu$, that is $h\leq h_\mu$, a contradiction with $h\geq h_\mu+\varepsilon$, that completes the proof of the lemma.
\end{proof}

Now we provide a necessary condition, using the value $h_\mu$ defined above, for the wave profile to converge to $1$ at $+\infty$ even when the delay $\tau$ becomes large. 
In order to state this, let
$g:[0,M_1]\to\R$ be a nondecreasing and $C^1$ function such that
\begin{equation*}
\begin{cases}
g(u)\leq f(u)\text{ for all $u\in [0,M_1]$},\\
\text{the function $G(u):=g(u)-u$ has exactly three zero $0<\theta^*<\gamma$ where $\theta<\theta^*$ and $\gamma <M_2$},\\
G'(0)<0,\; G'(\theta ^*)>0,\; G'(\gamma)<0\text{ and }
\int_0^\gamma G(u)du>0.
\end{cases}
\end{equation*}
Note that such a choice is possible due to Assumption \ref{ass:f-osc} $(ii)$. Due to Lemma \ref{lem:tw-monotone2} $(i)$, for each $h\in\R$, the nonlocal parabolic problem
\begin{equation}
\left(\partial_t -\partial_{xx}\right)w(t,x)=g\left(w(t,x-h)\right)-w(t,x),
\end{equation}
admits a travelling wave solution $\left(c_h,U_h\right)$ such that
$$
U_h'>0\text{ and }U_h(-\infty)=0,\;U_h(+\infty)=\gamma.
$$
Using this notation, our necessary condition reads as follows.

\begin{lemma}[Necessary condition for convergence to 1, even when the delay becomes large]\label{LE-conv} Assume $\mu=f'(1)<-1$.  Consider a sequence $\tau_n\to\infty$. Let us assume that the sequence of travelling wave solutions $\left(c_n,u_n\right)$ of \eqref{eq} with $\tau=\tau_n$ (as constructed in the proof of Theorem~\ref{th:construction}) satisfies
\begin{equation*}
\lim_{x\to +\infty} u_n(x)=1,\;\forall n\geq 0.
\end{equation*}
Then the sequence of positive numbers $h_n:=c_n\tau_n$ satisfies 
\begin{equation}\label{truc}
\limsup_{n\to\infty} h_n\leq h_\mu,
\end{equation}
and,  if $h\in [0,h_\mu]$ denotes an accumulation point of $\{h_n\}$ then it holds that
\begin{equation*}
c_h\leq 0.
\end{equation*}
\end{lemma}

\begin{proof}
Note that the upper bound \eqref{truc} for $h_n$ follows from Lemma \ref{LE-behaviour} and Lemma \ref{LE-location}. It remains to prove the second part of the lemma. To that aim assume, up to a subsequence, that
\begin{equation}\label{hyp-contra}
c_{n}\tau_n\to h\in [0,h_\mu].
\end{equation}
Observe that $c_n\to 0$ as $n\to\infty$.
Up to translation, we assume that $u_n$ is normalized so that
\begin{equation*}
u_n(0)=\frac{\theta}{2}.
\end{equation*} 
Next for each $n\geq 0$, we introduce $\varrho_n>0$ and $\sigma_n\in (\varrho_n+c_n\tau_n,\infty]$ defined by
\begin{equation*}
u_n\left(\varrho_n\right)=\theta,\;\;u'_n(x)>0,\;\forall x\in \left(-\infty,\sigma_n\right)\text{ and }u_n ' \left(\sigma_n\right)=0.
\end{equation*}
Observe that because of elliptic regularity, possibly along a subsequence, $u_n\to u$ as $n\to\infty$ locally uniformly where $u$  satisfies
\begin{equation*}
\begin{cases}
-u''(x)=f\left(u(x-h)\right)-u(x),\;\forall x\in \R,\\
u(0)=\frac{\theta}{2},\;u'(x)\geq 0,\;\forall x\in \left(-\infty,\sigma_\infty \right),
\end{cases}
\end{equation*}
wherein we have set $\sigma_\infty :=\displaystyle \liminf_{n\to\infty}\sigma_n \in (0,+\infty]$. By setting $\varrho_\infty =\displaystyle\liminf_{n\to\infty}\varrho_n$, it furthermore satisfies
\begin{equation*}
u\left(\varrho_\infty \right)=\theta\text{ if $\varrho_\infty <\infty$, and }u \leq \theta\text{ if }\varrho_\infty =+\infty. 
\end{equation*}
We shall show that the existence of such a function implies that $c_h\leq 0$. Our argument is now split into three steps.

\noindent {\bf First step:}  we show that $\varrho_\infty <\infty$. To that aim we argue by contradiction by assuming that $\varrho_\infty =\infty$. In that case one has $u\leq \theta$ and is nondecreasing on $\R$. Hence $u$ satisfies
\begin{equation*}
-u''(x)= f(u(x-h))-u(x)\leq u(x-h)-u(x)\leq 0,\;\forall x\in\R,
\end{equation*} 
and therefore it must be constant, which is clearly impossible.

\noindent {\bf Second step:} we show that 
\begin{equation}
\liminf_{x\to + \infty} u(x)>\theta.
\end{equation}
We consider both cases $\sigma_\infty =\infty$ and $\sigma_\infty <\infty$. In the former, then $u' \geq 0$ and hence $u(x)\to 1$ as $x\to + \infty$. In the latter, then because of Lemma~\ref{LE-localisation} one has for any $n\geq 0$:
\begin{equation*}
u_n(x)\geq M_2>\theta,\;\forall x\geq \sigma_n.
\end{equation*}
Passing to the limit $n\to \infty$ ensures that $\liminf_{x\to + \infty} u(x)\geq M_2$ and this completes the proof of the second step.

\noindent {\bf Third step:}  we conclude the proof of the lemma, that is we prove $c_h\leq 0$. Argue by contradiction by assuming that $c_h>0$ and
consider $w=w(t,x)$ the solution of the parabolic problem
\begin{equation}\label{PB_f-}
L[w](t,x):=\left[\partial_t-\partial_{xx}\right]w(t,x)-g\left(w(t,x-h)\right)+w(t,x)=0,
\end{equation}
supplemented with the initial datum $w_0$ defined by
\begin{equation*}
w_0(x)=\begin{cases} \theta+\delta &\text{ if }x\geq x_0,\\ 0 &\text{ if }x<x_0,\end{cases}
\end{equation*}
wherein $x_0$ is some large enough value and $\delta>0$ small enough so that $w_0(x)\leq u(x)$ for all $x\in\R$.
Next from $g\leq f$ we get 
\begin{equation*}
L[u]\geq 0 = L[w],\;\forall (t,x)\in (0,\infty)\times\R,
\end{equation*}
so that the comparison principle  (recall that the function $g$ is nondecreasing) 
 ensures
\begin{equation}\label{ineq-comp}
w(t,x)\leq u(x),\;\forall (t,x)\in (0,\infty)\times\R.
\end{equation}
Now recalling the stability result stated in Lemma \ref{lem:tw-monotone2} $(ii)$, one has
\begin{equation*}
U_h\left(x+c_ht+\xi\right)-qe^{-k t}\leq w(t,x)\leq u(x), 
\end{equation*}
for some constant $\xi\in\R$, $q>0$ and $k>0$.
Plugging $x=0$ and using the normalisation condition $u(0)=\frac{\theta}{2}$ yield
\begin{equation*}
U_h\left(c_ht+\xi\right)-qe^{-k t}\leq \frac{\theta}{2}. 
\end{equation*}
If $c_h>0$, letting $t\to\infty$ implies that $
\gamma=\lim_{x\to + \infty}U_h(x)\leq \frac{\theta}{2}$.
Recalling that $\gamma>\theta$, we have reached a contradiction. 
\end{proof}

We are now in the position to complete the  proof of Theorem \ref{th:large-delay}. 

\begin{proof}[Proof of Theorem \ref{th:large-delay}]
In view of Lemma \ref{LE-conv}, to complete the proof of this result, it is sufficient to show that when $\mu=f'(1)$ is large enough (in absolute value) then
\begin{equation*}
c_{h}>0,\;\forall h\in [0,h_\mu].
\end{equation*}
Recalling now Lemma \ref{lem:tw-monotone2} $(iii)$ and its subsequent Remark \ref{REM1}, it is therefore sufficient to prove that
\begin{equation*}
\lim_{\mu\to -\infty} h_\mu=0,
\end{equation*}
or in other words that $\lim_{\mu \to -\infty} \kappa_\mu = + \infty$. To that aim let $\lambda_\mu \in S_0$ be given such that
\begin{equation*}
\kappa_\mu=\left|\frac{1-\mu e^{-\lambda_\mu}}{\lambda_\mu^2}\right|,\;\forall \mu<-1.
\end{equation*}
Note indeed that the infimum in the definition of $\kappa_\mu$ (see Lemma~\ref{LE-location}) must be attained.
Next consider a sequence $\mu_n\to-\infty$ such that
\begin{equation*}
\lim_{n\to\infty}\kappa_{\mu_n}=\liminf_{\mu \to-\infty} \kappa_\mu,
\end{equation*}
and let us show that this limit is $+\infty$. This means that the sequence $\left\{\kappa_{\mu_n}\right\}$ is unbounded. 

To reach this goal, let us argue by contradiction by assuming that $\kappa_{\mu_n}$ is bounded. Next let us observe that the sequence $\{\lambda_{\mu_n}\}$ is bounded. Indeed one has for all $n\geq 0$
\begin{equation}\label{pml}
\left|\lambda_{\mu_n}\right|^2e^{\Re\left(\lambda_{\mu_n}\right)}\kappa_{\mu_n}=\left|e^{\lambda_{\mu_n}}-\mu_n\right|.
\end{equation}
This implies that 
\begin{equation*}
\liminf_{n\to\infty}\Re\left(\lambda_{\mu_n}\right)>-\infty,
\end{equation*}
and thus the sequence $\left\{\lambda_{\mu_n}\right\}$ is bounded. From this boundedness property, it follows from \eqref{pml} that the sequence $\left\{\kappa_{\mu_n}\right\}$ is unbounded, a contradiction.

We conclude that $\kappa_\mu\to +\infty$ as $\mu\to-\infty$, which ends the proof of Theorem \ref{th:large-delay}.
\end{proof}

\subsection{Convergence to a wavetrain}\label{ss:wavetrain}

This short subsection is devoted to the proof of the fact stated in Remark \ref{rem:wavetrain}: under the additional assumption $M_2>\beta$ (see the beginning of subsection \ref{ss:osci}), we can actually prove convergence to a wavetrain. We start with standard definitions.

\begin{definition}[Omega limit orbit, set] 
A function $u\in C^2(\R)$ is said to be an omega limit orbit of $U$ if there exists a sequence $x_n\to \infty$ as $n\to\infty$ such that
$$
u(x)=\lim_{n\to\infty} U(x+x_n),\text{ locally uniformly for $x\in\R$}.
$$
The set of all omega limit orbits of $U$ is called the omega limit set of $U$ and it is denoted by $\omega(U)$. 

Similarly we define the omega limit set, $\omega(u)$, for any orbit $u\in\omega(U)$.
\end{definition}

 {}From  Lemma \ref{LE-localisation}, one may first notice that any $u\in\omega(U)$ has to satisfy $M_2\leq u\leq M_1$. Next, for each $u\in\omega(U)$, we define, for each $x\in\R$, $\Psi[u]_x\in C^0(\mathbb K)$ by 
\begin{equation*}
\Psi[u]_x(\theta):=\begin{cases} u(x+\theta)-1 &\text{ if $\theta\in [-h,0]$},\\ u'(x) &\text{ if $\theta=1$}. \end{cases}
\end{equation*}
Here it is easy to check that 
\begin{equation*}
u\equiv 1\;\Leftrightarrow\;\exists x\in\R,\;  \forall \theta\in\mathbb K,\;   \Psi_x[u](\theta)=0. 
\end{equation*}
Moreover, due to \eqref{sign-changes} and the results in \cite{Malet-Paret}, one obtains that if $u\not\equiv 1$ then ${\rm sc}\,\left(\Psi[u]_x\right)\in\{0,1,2\}$ for all $x\in\R$ and, if there exists some $x_0\in\R$ such that ${\rm sc}\,\left(\Psi[u]_x\right)=0$  then it remains equal to~$0$ for all $x\geq x_0$. The latter situation cannot occur. Indeed in that case, $u=u(x)$ should be, for $x$ large enough, increasing and above $1$ or decreasing and below $1$. We have already shown that these situations are both impossible.
This means that when $u\not\equiv 1$ then ${\rm sc}\,\left(\Psi[u]_x\right)\in\{1,2\}$ for all $x\in\R$.

In view of these remarks, the proof of Lemma \ref{LE-behaviour} directly adapts to get that, if $u \not \equiv  1$, then it may not converge to 1 at $+\infty$. In other words, we have the following.

\begin{corollary}\label{CORO}
Let Assumption \ref{ass:f-osc} hold. Under the assumptions of Lemma \ref{LE-behaviour}, any orbit $u=u(x)$ in $\omega(U)$ satisfies the following alternative:
either $u\equiv 1$ or $\omega(u)\neq \{1\}$. Moreover one has, for each $u\in\omega(U)$,
\begin{equation*}
u \not\equiv 1\,\Leftrightarrow\; \begin{cases}\Psi[u]_x\in C^0(\mathbb K)\setminus \{0\},\\ {\rm sc}\,\left(\Psi[u]_x\right)\in\{1,2\},\end{cases}\forall x\in\R.
\end{equation*}
\end{corollary}
In particular, the above corollary ensures that $\omega(U)$ does not contain any non-constant orbit converging to $1$ at $+\infty$ and thus $\omega(U)$ does not contain any nontrivial homoclinic orbit to $1$.

We can now state the result of convergence to a wavetrain.

\begin{corollary}\label{wave-train}
Let Assumption \ref{ass:f-osc} hold and suppose further that $M_2>\beta$. 
Under the assumptions of Lemma \ref{LE-behaviour}, there exists a non-constant periodic wavetrain solution $U_*$ for the wave speed~$c$, namely a solution of 
\begin{equation*}
-U_*''(x)+cU_*'(x)=f(U_*(x-h))-U_*(x),\quad  U_*(x+T)=U_*(x) \text{ for some $T>0$},\quad  x\in \R,
\end{equation*}
such that $M_2\leq U_*\leq M_1$, ${\rm sc}\,\left(\Psi[U_*]_x\right)\in\{1,2\}$ for all $x\in\R$ and such that $\omega(U)$ consists in the single periodic orbit $U_*$, in the sense that $\omega(U)=\{U_*(\cdot +p),\;p\in\R\}$.
\end{corollary}

\begin{proof}
Note that the additional assumption $M_2>\beta$ allows us to obtain, using Lemma \ref{LE-localisation}, that $U(x)$ is trapped between $\beta$ and $M_1$ for all $x\geq \sigma_*$. Hence, the function $g=g(z_0,z_1,z_2)$, defined in \eqref{system-z},  satisfies the condition
\begin{equation*}
\frac{\partial g}{\partial z_0}\equiv 1,\quad \frac{\partial g}{\partial z_2}\left(z_0(t),z_1(t),z_0(t-h)\right)=-f'(U(t-h))>0, \text{ for all $t\geq \sigma_*+h$},
\end{equation*}
in view of Assumption \ref{ass:f-osc}. Hence, the omega limit set of $(z_0,z_1)$, or equivalently $\omega(U)$, satisfies the Poincar\'e-Bendixson dichotomy, as stated in \cite[Theorem 2.1]{MP-S}. This reads, in our context, as either $\omega(U)$ consists in a single non-constant periodic orbit or, each orbit in $\omega(U)$ is homoclinic to $1$.
However, since $\omega(U)\neq \{1\}$ (because $U$ does not converge to $1$) and since it does not contain any nontrivial orbit homoclinic to $1$ from the previous corollary, it therefore consists in a single periodic orbit. This proves the corollary.
\end{proof}

\end{document}